\documentclass[11pt]{amsart}
\usepackage[margin=1in]{geometry}
\usepackage[T1]{fontenc}
\usepackage{amsmath,amssymb,amsthm,mathtools,mathrsfs}
\usepackage{microtype,graphicx,nicefrac}
\usepackage[dvipsnames]{xcolor}
\usepackage[square,authoryear]{natbib}
\usepackage[citecolor=PineGreen,colorlinks=true,linkcolor=RoyalBlue,urlcolor=BrickRed]{hyperref}
\hypersetup{pdftitle={Sandpile avalanches in two dimensions are finite},pdfauthor={Ahmed Bou-Rabee}}
\newtheorem{theorem}{Theorem}[section]
\newtheorem{proposition}[theorem]{Proposition}
\newtheorem{lemma}[theorem]{Lemma}
\numberwithin{equation}{section}
\newcommand{\Z}{\mathbb Z}
\newcommand{\PP}{\mathbb P}
\newcommand{\EE}{\mathbb E}
\newcommand{\nf}{\nicefrac}
\newcommand{\cL}{\mathscr{L}}
\DeclareMathOperator{\diam}{diam}
\title[Sandpile avalanches in two dimensions]{Sandpile avalanches in two dimensions are finite}
\author{Ahmed Bou-Rabee}
\begin{document}
\raggedbottom
\begin{abstract}
In the abelian sandpile model on~$\Z^d$, a site with at least~$2d$ grains topples by sending one grain to each neighbor. In a box~$[-N,N]^d$ whose boundary absorbs grains, adding grains at random sites and toppling defines a Markov chain. \citet{AthreyaJarai2004} constructed the stationary state on~$\Z^d$ as the limit of its stationary distributions as~$N\to\infty$. They asked whether, when a grain is added at the origin of a sandpile in the stationary state, the resulting avalanche of topplings is finite. \citet{JaraiRedig2008} proved that the answer is yes for~$d\geq3$, using the transience of random walk. We prove, using a different method, that avalanches are also finite for~$d=2$.
\end{abstract}
\maketitle
\begin{figure}[!h]
\centering
\includegraphics[width=0.71\textwidth]{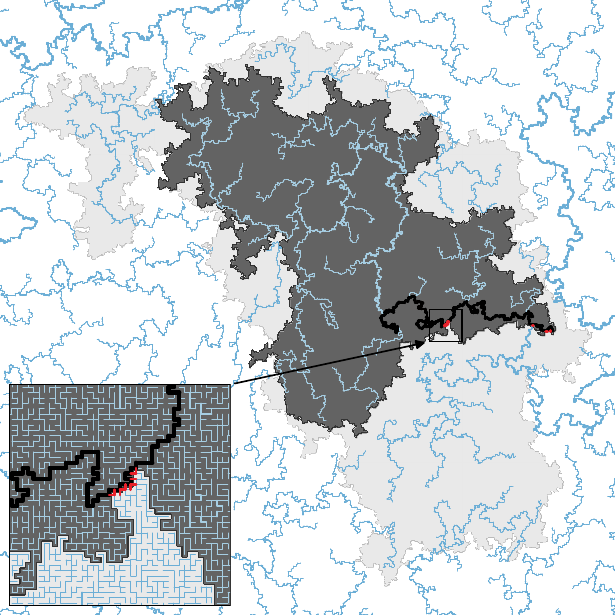}
\caption{An avalanche, in light gray, has one wave for each toppling of the origin, and one of these waves is shown in dark gray. The sandpile left by the wave is described exactly by two spanning trees, in blue: one of the wave and one of the rest of the box. Joining these two trees along a red edge into one spanning tree of the box lets us control large avalanches by a single spanning tree quantity, as explained in Section~\ref{ssec:proof-overview}.}
\label{fig:avalanche}
\end{figure}
\clearpage

\section{Introduction}\label{sec:introduction}

Consider a graph with a distinguished vertex called the \textbf{sink}. A \textbf{sandpile}~$\eta$ on the graph is a function from the set of non-sink vertices to the integers, where~$\eta(v)$ represents the number of grains, or \textbf{height}, at vertex~$v$. A sandpile~$\eta$ is \textbf{stable} if~$\eta(v)$ is less than the degree of~$v$ for every non-sink vertex~$v$. We call a vertex~$v$ \textbf{unstable} in~$\eta$ if~$v$ is not the sink and~$\eta(v)$ is at least the degree of~$v$. An unstable vertex~$v$ can \textbf{topple}, moving one grain along each edge from~$v$. Grains sent to the sink disappear. The topplings that follow the addition of a grain to a stable sandpile form an \textbf{avalanche} (Figure~\ref{fig:avalanche}). A box~$[-N,N]^d$ of~$\Z^d$ with \textbf{wired boundary conditions} is the graph obtained from~$\Z^d$ by identifying all vertices outside the box to a single vertex, the sink, and deleting the loops at the sink. On this graph every sandpile becomes stable after finitely many topplings, and adding grains at uniformly chosen vertices and toppling until the sandpile is stable defines a Markov chain on stable sandpiles.

\citet{Dhar1990} showed that, when a grain is added at a vertex of a sandpile in the stationary distribution of this chain, the expected number of topplings at each vertex is proportional to the Green function of simple random walk in the box between the two vertices. As the box grows, this Green function stays bounded when~$d\geq3$ and grows logarithmically when~$d=2$~\citep{LawlerLimic2010}. \citet*{AthreyaJarai2004,AthreyaJarai2006} constructed the stationary state on~$\Z^d$ as the limit of these stationary distributions and asked whether avalanches are almost surely finite, expecting the transience or recurrence of random walk to play a role. \citet*{JaraiRedig2008} proved that the answer is yes when~$d\geq3$. When~$d=2$, \citet*{BhupatirajuHansonJarai2017} showed that the expected radius of an avalanche is infinite, but whether avalanches on~$\Z^2$ are almost surely finite has remained open~\citep*{JaraiRedig2008,JaraiWerning2014,Jarai2018}.

\subsection{Main result}\label{ssec:main-results}

In a box of~$\Z^2$ with wired boundary conditions, a stable sandpile~$\eta$ is \textbf{recurrent} if~$\eta$ can be obtained from every sandpile by adding grains and toppling unstable vertices. The stationary distribution of the Markov chain above is uniform on recurrent sandpiles~\citep{Dhar1990}. Let~$\mu$ be the \textbf{stationary state} on~$\Z^2$, the weak limit of these uniform measures as the box grows~\citep*{AthreyaJarai2004}. The measure~$\mu$ is translation invariant.

For a sandpile~$\eta$, let~$\eta+\delta_0$ be the sandpile obtained by adding one grain at the origin. A toppling at~$v$ is \textbf{legal} if~$v$ is unstable in the current sandpile when the toppling occurs. If a finite sequence of legal topplings takes~$\eta+\delta_0$ to a stable sandpile, then~$\eta+\delta_0$ \textbf{stabilizes}. By the abelian property~\citep{Dhar1990}, the number~$u(x)$ of topplings at~$x$ does not depend on the sequence. We call~$u$ the \textbf{odometer function} of the stabilization. The \textbf{radius} of the avalanche is~$\rho\coloneqq\max\{|x|_\infty:u(x)>0\}$, with~$\rho=0$ when no vertex topples and~$\rho=\infty$ when~$\eta+\delta_0$ does not stabilize.

\begin{theorem}\label{th:main}
For $\mu$-almost every sandpile~$\eta$, the sandpile~$\eta+\delta_0$ stabilizes. Moreover, there are universal constants~$a>0$ and~$C<\infty$ such that every~$M\geq0$ satisfies
\begin{equation*}
\mu(\rho>M)\leq C(\log(M+2))^{-a}\,.
\end{equation*}
\end{theorem}

Theorem~\ref{th:main} shows that avalanches on~$\Z^2$ are almost surely finite and so answers the question of~\citet*{AthreyaJarai2004} for~$d=2$.

\subsection{Proof overview}\label{ssec:proof-overview}

Our proof begins with the decomposition of an avalanche in a box into \textbf{waves}, due to the physicists~\citet*{IvashkevichKtitarevPriezzhev1994}. While the origin is unstable, topple the origin once and then stabilize the other vertices, treating the origin as a sink; the vertices that topple form a wave. These topplings are legal, so when~$\eta+\delta_0$ stabilizes there are exactly~$u(0)$ waves, each finite, and each vertex topples at most once in each wave. Unlike an avalanche, a wave has an exact description. In a box, the sandpiles that waves leave behind in the stationary state correspond bijectively to two-component spanning forests, with the origin and the sink in different components, and the wave is the component of the origin.

Parts of this picture were made rigorous by~\citet*{JaraiRedig2008}, who used waves to prove that avalanches are finite when~$d\geq3$, and by~\citet*{BhupatirajuHansonJarai2017}, who showed that each two-component forest corresponds to exactly one wave of exactly one avalanche. An avalanche that reaches distance~$M$ contains a wave that reaches distance~$M$, so the probability of such an avalanche is at most the expected number of such waves. By this correspondence, the expected number of such waves equals the expected number of waves times the probability that the component of the origin in a uniformly chosen two-component forest of the box reaches distance~$M$. The expected number of waves is the expected number of topplings at the origin, which by the formula of Dhar recalled above grows like the logarithm of the size of the box. This logarithm is the main difficulty.

We overcome this difficulty by counting only the waves that a large avalanche must contain, and we bound their number by an expectation over a uniform spanning tree of the box instead of a sum over two-component forests. The proof proceeds in three steps.
\begin{enumerate}
\item Fix integers~$r\geq2$ and~$M\geq r+2$, and mark the origin, its four neighbors, and the vertex~$(r,0)$. We prove the following bound on the radius of an avalanche:
\begin{equation}\label{ov:reduction}
\mu(\rho>M)\leq\limsup_{N\to\infty}\EE\bigl[\cL_N^{-1}\bigr]+Cr(M-r)^{-\nf13}\,,
\end{equation}
where, in a uniformly chosen spanning tree of~$[-N,N]^2$ oriented toward the sink, $\cL_N$ is the number of lattice edges from the tree paths between the marked vertices to vertices whose own paths to the sink avoid these tree paths. These are the red edges in Figures~\ref{fig:avalanche} and~\ref{fig:joining}. The main new idea behind the proof of~\eqref{ov:reduction} is to control a large avalanche by the transitions from a wave that contains every marked vertex to a wave that misses one. Joining a spanning tree of a wave and a spanning tree of the rest of the box into a spanning tree of the whole box then turns the number of these transitions into an expectation over a single tree.
\begin{figure}[!t]
\centering
\includegraphics[width=\textwidth]{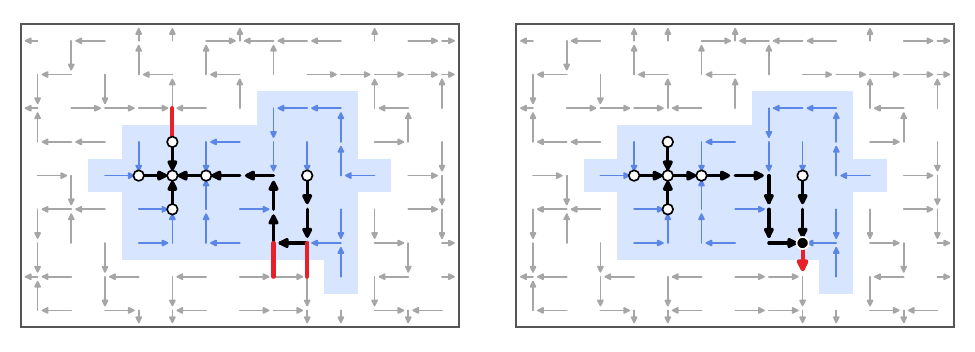}
\caption{Left: the sandpile left by a shaded wave gives a spanning tree of the wave, in blue, and a spanning tree of the rest of the box, in gray. The white discs are the marked vertices, and the~$\cL_N=3$ red edges go from the black paths between them to the rest of the box. Right: joining the two trees along one red edge gives a spanning tree of the box in which the wave is the set of vertices whose paths to the sink pass through the black disc, the first vertex that the paths from the marked vertices to the sink all share.}
\label{fig:joining}
\end{figure}
\item By~\eqref{ov:reduction}, we must show that~$\cL_N$ is usually large. Let~$\ell$ be a large integer, and suppose that the tree paths between the marked vertices reach distance~$2\ell$ from the first vertex that the paths from the marked vertices to the sink all share. We first prove that
\begin{equation}\label{ov:one-scale}
\PP\bigl(\cL_N\geq\tfrac18\log \ell\bigr)\geq\varepsilon
\end{equation}
for a universal constant~$\varepsilon>0$, conditionally on the tree paths between the marked vertices and on the path from this vertex to the sink.
\item Applying~\eqref{ov:one-scale} at many scales, each much larger than the previous one, we prove that
\begin{equation}\label{ov:multiscale}
\EE\bigl[\cL_N^{-1}\bigr]\leq C(\log r)^{-a}
\end{equation}
for a universal constant~$a>0$ and every~$N\geq r$. The scales are not independent, but~\eqref{ov:one-scale} still holds, with smaller constants, given the earlier scales. Taking~$r=\lfloor M^{\nf16}\rfloor$ in~\eqref{ov:reduction} and~\eqref{ov:multiscale} proves Theorem~\ref{th:main}.
\end{enumerate}
\medskip

\noindent\textbf{Step 1.} We prove the bound~\eqref{ov:reduction} in Proposition~\ref{sp:reciprocal-reduction}. The two terms on the right side of~\eqref{ov:reduction} correspond to two ways in which an avalanche can be large. The origin is stable after the last wave, so the last wave misses a neighbor of the origin. Hence an avalanche of radius larger than~$M$ has either a wave that contains every marked vertex immediately followed by a wave that misses a marked vertex, or a wave that misses a marked vertex and is not contained in~$[-M,M]^2$.

The first term in~\eqref{ov:reduction} bounds the expected number of these pairs of consecutive waves, and the second term bounds the expected number of waves that miss a marked vertex and are not contained in~$[-M,M]^2$. The second term is the easier one. Adding one edge to the two-component spanning forest of such a wave gives a spanning tree of the box in which the wave is the set of vertices whose paths to the sink pass through a single vertex that is a neighbor of the origin or lies on the segment from the origin to~$(r,0)$. A bound of~\citet*{BarlowMasson2011} on paths in the uniform spanning tree of~$\Z^2$ shows that, in a uniform spanning tree of the box, the set of vertices whose paths to the sink pass through any one such vertex leaves~$[-M,M]^2$ with probability at most~$C(M-r)^{-\nf13}$, and summing over these~$r+4$ vertices gives the second term.

We now explain the bound on the first term. Fix a wave that contains every marked vertex, the first wave of a pair. Adapting~\citet*{JaraiWerning2014}, we encode the sandpile left by the first wave by two spanning trees, one of the first wave oriented toward the origin and one of the rest of the box oriented toward the sink. The next wave starts at the origin and spreads outward along the first tree, and this spreading can stop only at vertices next to the outside of the first wave. So the next wave misses a marked vertex only if the tree path from that marked vertex to the origin passes next to the outside of the first wave. Figure~\ref{fig:joining} shows in red the lattice edges where this happens.

Joining the two trees along one red edge gives a spanning tree of the box in which the red edges are exactly the edges counted by~$\cL_N$. A pair of trees with~$\cL_N$ red edges gives~$\cL_N$ spanning trees of the box in this way, and distinct pairs give distinct trees. Hence the expected number of these pairs of consecutive waves is at most~$\EE[\cL_N^{-1}]$. Figure~\ref{fig:avalanche} shows the two trees and the red edges for one wave of a simulated avalanche, with~$r=200$: the wave contains every marked vertex, the black paths are the tree paths between the marked vertices, the~$\cL_N=14$ red edges go from these paths to the rest of the box, and the next wave misses the marked vertex~$(r,0)$.

\medskip
\noindent\textbf{Step 2.} We prove the lower bound~\eqref{ov:one-scale} in Lemma~\ref{rd:conditional-law} and Proposition~\ref{fo:local-comparison}: at the scale~$\ell$, the count~$\cL_N$ is at least~$\tfrac18\log \ell$ with probability at least~$\varepsilon$, whatever the conditioned paths are. To prove this bound, we use the Paley--Zygmund inequality, which requires a lower bound on the mean of~$\cL_N$ and an upper bound on the variance of~$\cL_N$.

To bound the mean of~$\cL_N$ from below, we condition on the tree paths between the marked vertices and on the path to the sink from the first vertex~$x$ that all the paths from the marked vertices to the sink share. The path from another vertex~$y$ to the sink then avoids these tree paths with the probability that simple random walk from~$y$ hits the vertices that follow~$x$ on the path from~$x$ to the sink before these tree paths~\citep{Wilson1996}. Summing this probability over the edges from these tree paths gives the conditional mean of~$\cL_N$, and an electrical network estimate shows that this sum is at least~$\tfrac12\log \ell$.

To bound the variance of~$\cL_N$ from above, we must show that the counted edges are not so strongly correlated that~$\cL_N$ is large only on a rare event. Passing to a finite domain around~$x$ replaces~$\cL_N$ by a count that~$\cL_N$ stochastically dominates. There, planar duality turns this count into the number of edges that one random path, in a uniform spanning tree of the dual graph, shares with a fixed path along the boundary of the domain. The identities of~\citet*{KenyonWilson2009} give the probability that this random path uses two given boundary edges, and they reduce the variance bound to bounds on the Green function of the dual graph at points of the boundary path. We bound these quantities, on average along the boundary path, by a constant times~$\log \ell$ for every domain, using the square tiling of~\citet*{BrooksSmithStoneTutte1940} in the form of~\citet{Lovasz2019}, in which every domain becomes a set of non-overlapping squares in a strip. Since the mean of this count is at least~$\tfrac12\log \ell$, its variance is at most a constant times the square of its mean.

\medskip
\noindent\textbf{Step 3.} We prove the inverse moment bound~\eqref{ov:multiscale} in Proposition~\ref{ct:reciprocal} by iterating~\eqref{ov:one-scale}. Since~$\cL_N^{-1}\leq1$, we have~$\EE[\cL_N^{-1}]\leq\PP(\cL_N<\theta)+1/\theta$ for every~$\theta>0$, so one scale gives only~$\EE[\cL_N^{-1}]\leq1-\varepsilon+o(1)$. We therefore use order~$\log\log r$ scales below~$r$, each a fixed large power of the previous one.

We build the tree scale by scale with Wilson's algorithm~\citep{Wilson1996}, which grows the tree from loop-erased random walks, and we apply~\eqref{ov:one-scale} at each scale given the earlier scales. If~\eqref{ov:one-scale} held at each scale~$\ell$, then, since the number of counted edges found up to~$\ell$ increases with~$\ell$, the conditional mean of the inverse of this number would be at most~$1-\varepsilon$ times the inverse at the previous scale plus~$8/\log\ell$, and iterating this bound would give~\eqref{ov:multiscale}.

The difficulty is that the branches built at earlier scales become additional roots, and~$\cL_N$ does not count the edges found at a later scale that start on these branches. By the discrete Beurling estimate of~\citet{Kesten1987}, these branches stay close to~$x$ except on an event of probability at most a constant times~$\ell^{-\nf12}$ at the scale~$\ell$, and an electrical network estimate shows that they take only a small fraction of the edges found at the next scale. Hence~\eqref{ov:one-scale} still holds at each scale with smaller constants outside this event, and~\eqref{ov:multiscale} follows.

The number of scales is limited by how far apart consecutive scales must be. At a scale~$\ell$, the bound~\eqref{ov:one-scale} provides only order~$\log\ell$ counted edges. The branches built at earlier scales may reach a distance that is a fixed power of the previous scale, and in mean they take up to a constant times the logarithm of this distance of the edges counted at~$\ell$. For the remaining edges to still be of order~$\log\ell$, the logarithm of each scale must be a large constant times the logarithm of the previous scale, so there are only order~$\log\log r$ scales below~$r$. Each scale contributes a factor~$1-\varepsilon$, and these factors multiply to a negative power of~$\log r$, which is~\eqref{ov:multiscale}. This is why Theorem~\ref{th:main} bounds~$\mu(\rho>M)$ by a negative power of~$\log M$ rather than of~$M$.

\subsection{Related work}\label{ssec:related-work}

\subsubsection{Physics literature}
\citet*{BakTangWiesenfeld1987} introduced the sandpile model as an example of self-organized criticality: in their simulations, avalanche sizes followed a power law, as at the critical point of a phase transition, but without the tuning of a parameter such as the temperature. \citet*{PriezzhevKtitarevIvashkevich1996} used the waves of Section~\ref{ssec:proof-overview} to predict the exponents of these power laws on~$\Z^2$. Assuming that successive waves are nested with decreasing sizes, they predicted that exactly~$n$ vertices topple in an avalanche with probability of order~$n^{-\nf54}$. If the number of toppled vertices is of order the square of the radius, this prediction says that the radius of an avalanche exceeds~$M$ with probability of order~$M^{-\nf12}$, which is much smaller than the bound of Theorem~\ref{th:main}.

However, waves need not be nested~\citep{DharManna1994}. Each wave repeats the topplings of the preceding wave up to the boundary of the preceding wave~\citep{PriezzhevKtitarevIvashkevich1996}, but in simulations \citet*{PaczuskiBoettcher1997} found that a wave is more often larger than the preceding wave, and~\citet*{KtitarevPriezzhev1998} observed that even a smaller wave need not lie inside its predecessor. Single waves remain well described by their two-component spanning forests, from which~\citet*{KtitarevLubeckGrassbergerPriezzhev2000} developed scaling predictions for waves. Complete avalanches are more complex: \citet*{DeMenechStella2000} found correlations between the sizes of successive waves, and~\citet*{TebaldiDeMenechStella1999} found numerical evidence that, unlike the number of toppled vertices, the number of topplings is not described by a single exponent. \citet*{JaraiSun2019} and~\citet{Sun2022} studied toppling probabilities and the number of waves numerically.

Numerical work also suggests that the frontiers of avalanches~\citep*{Saberi2009} and the boundaries of waves~\citep*{AzimiTafreshi2010,DashtiNaserabadi2012} are described by the Schramm--Loewner evolution with parameter~$2$, the scaling limit of loop-erased random walk~\citep*{LawlerSchrammWerner2004}.

\subsubsection{Infinite-volume sandpiles}
The burning bijection of~\citet*{MajumdarDhar1992} maps the recurrent sandpiles of a finite graph bijectively to the spanning trees of the graph, and the construction of~$\mu$ by~\citet*{AthreyaJarai2004} combines this bijection with the uniform spanning forest of~$\Z^d$, whose components are one-ended~\citep{Pemantle1991,BLPS2001}. This approach extends to other graphs. \citet{Jarai2012TransitiveGraphs} used a coupling to construct the stationary state on certain transitive graphs and on infinite graphs whose wired uniform spanning forest is connected and one-ended, and~\citet{JaraiWerning2014} removed the connectedness assumption when every component is one-ended. \citet*{GamlinJarai2014} constructed anchored burning bijections on such graphs, and on~$\Z^d$ they proved a power-law bound on the rate of convergence to the stationary state.

\subsubsection{Rigorous bounds on avalanches}
Away from two dimensions, more is known. For~$d\geq5$, \citet{Hutchcroft2020} determined the critical exponents of the radius and of the total number of topplings, and~\citet*{LiLiuShiraishi2026} recently announced these exponents for~$d=3$. On supercritical Galton--Watson trees conditioned to survive, \citet*{JaraiRuszelSaada2020} showed that an avalanche has more than~$n$ topplings with probability at most of order~$n^{-\nf12}$, and at least of this order when the offspring law has an exponential moment.

In two dimensions, the rigorous results come from the analysis of single waves. Besides the correspondence described in Section~\ref{ssec:proof-overview}, \citet*{BhupatirajuHansonJarai2017} showed that the last wave of an avalanche in a box is distributed as the component of the origin in a uniformly chosen two-component spanning forest, conditioned on a neighbor of the origin lying in the component of the sink, and that the~$k$th wave from the end, conditionally on its existence, converges in distribution to an almost surely finite set as the box grows. They also proved inequalities for the critical exponents of avalanches on~$\Z^d$ and bounded the radius of avalanches with a bounded number of waves.

\subsection{Notation and conventions}

\begin{itemize}
\item We write~$|\cdot|_\infty$ for the maximum norm on the plane and~$\diam_\infty$ for the diameter in this norm.
\item For an integer~$r\geq0$, we also write~$[-r,r]^2$ for the set of points~$y\in\Z^2$ with~$|y|_\infty\leq r$, and we let~$\partial[-r,r]^2\coloneqq\{y\in\Z^2:|y|_\infty=r\}$. For~$x\in\Z^2$, we write~$x+[-r,r]^2$ and~$x+\partial[-r,r]^2$ for the translates of these sets by~$x$.
\item For a vertex~$x$ and a set~$V$ of vertices, let~$\deg(x,V)$ be the number of edges from~$x$ to~$V$.
\item We write~$\PP_y$ for the law of simple random walk from a vertex~$y$ on~$\Z^2$ or on a finite network, and the position at time zero counts as a visit.
\item The constant~$C<\infty$ may change from line to line.
\end{itemize}

\subsection{Formalization}\label{ssec:formalization}

A Lean~4 formalization accompanying this paper is maintained at
\begin{center}
\url{https://github.com/nitromannitol/finite-avalanche}
\end{center}
The development uses Mathlib together with explicitly listed standard results from the literature. The repository records the exact statements and sources of these inputs and the correspondence between the paper and the Lean declarations.

\subsection*{Use of artificial intelligence}

The results of this paper were developed in collaboration with GPT-5.5 and models from the GPT-6 family from late May through early October 2026. The author directed the investigation and contributed ideas. The models contributed mathematical ideas, proof attempts, computations, counterexamples, and revisions. Claude Opus~5.5 made the figures and helped with the exposition.

We began by looking for exact maps between the waves of an avalanche, hoping to describe the complete avalanche through the two-component spanning forests of single waves. With the models, we enumerated all sandpiles in small boxes to test the proposed maps, including maps between the first and last waves. These computations ruled out every proposed identification of the complete avalanche with a single forest component: successive waves need not be nested, and the proposed maps did not preserve the uniform measure.

We then looked for a way to count complete avalanches instead of single waves. The author proposed the idea behind the main reduction: controlling large avalanches through certain pairs of their waves. The models proposed candidate bijections and tested them by exact enumeration, and this search led to a bound on the probability of a large avalanche by a quantity defined through the uniform spanning tree, in which the number of waves no longer appears.

We then had to show that this quantity tends to zero. The models proposed proof routes and found counterexamples to several of them: estimates of means were not enough, and the author saw that the proof needed a probability bound that survives revealing part of the tree. We first used such a bound to prove that avalanches are finite, and then applied the bound at many scales to obtain the logarithmic tail bound.

Throughout the investigation, separate agents pursued competing ideas and checked proposed proofs. We recorded each failed approach with the reason for its failure, so that later attempts did not repeat the approach. The author checked the proofs and wrote the paper, and takes responsibility for its mathematical content and presentation.

\subsection*{Acknowledgments}

The author thanks Tom Hutchcroft, Xinyi Li, Antal J\'arai, Christoforos Panagiotis, and Runsheng Liu for useful conversations.

\section{Preliminaries}\label{sec:prelim}

\subsection{Spanning trees}\label{ssec:prelim-trees}

For a finite multigraph~$G$, let~$\operatorname{SpanT}(G)$ be the set of spanning trees of~$G$. Given a set of roots, in~$G$ or in~$\Z^2$ with finite complement, a \textbf{rooted spanning forest} is an acyclic spanning subgraph in which every component contains exactly one root. Identifying the roots to one vertex and deleting the resulting loops turns rooted spanning forests into the spanning trees of a finite multigraph.

For an integer~$N\geq1$, we also write~$[-N,N]^2$ for the box with wired boundary conditions and~$s_N$ for its sink. Let~$T_N$ be a uniform spanning tree of~$[-N,N]^2$ oriented toward~$s_N$. A sequence of laws of random sets of edges, or of random sandpiles, \textbf{converges weakly} to a law if the probability of every event that depends on finitely many edges, or on finitely many heights, converges to the probability of the event under that law. As~$N\to\infty$, the law of~$T_N$ converges weakly to the law of the \textbf{uniform spanning tree}~$T$ of~$\Z^2$~\citep[Theorem~2.3 and the proof of Theorem~2.4]{Pemantle1991}, which is almost surely a spanning tree of~$\Z^2$~\citep[proof of Theorem~4.2]{Pemantle1991} with one end~\citep[Theorem~12.4]{BLPS2001}.

In a finite tree oriented toward a root, the \textbf{oriented path} from a vertex~$x$ follows the outgoing edges from~$x$, and the \textbf{past} of~$x$ is the set of vertices whose oriented paths pass through~$x$, written~$\operatorname{Past}_{T_N}(x)$ in~$T_N$. The oriented paths from a finite nonempty set eventually coincide, and their first common vertex is the \textbf{common ancestor} of the set. Since~$T$ has one end, removing a vertex~$x$ from~$T$ leaves exactly one infinite component. The outgoing edge of~$x$ in~$T$ goes from~$x$ to this component, and the past~$\operatorname{Past}_T(x)$ is the union of~$x$ and the other components, which are finite.

Fix a finite nonempty set~$S\subset\Z^2$. For a tree~$t$ in~$[-N,N]^2$ that contains~$S$, the oriented paths from~$S$ up to their common ancestor form the \textbf{minimal subtree}~$\operatorname{MinSt}(t,S)$ of~$t$ containing~$S$, and~$\operatorname{MinSt}(t,S)$ also denotes the vertex set of this subtree. If the common ancestor is not~$s_N$, let~$\cL_t$ be the number of lattice edges, in~$t$ or not, from~$\operatorname{MinSt}(t,S)$ to vertices outside the past of the common ancestor, counting an edge to~$s_N$ as a lattice edge leaving~$[-N,N]^2$; otherwise let~$\cL_t\coloneqq\infty$. Unless the common ancestor is the root of~$t$, the outgoing edge of the common ancestor is counted, so~$\cL_t\geq1$.

For~$M\geq0$ and~$N\geq M$, the event that the past of a vertex~$x$ lies in~$[-M,M]^2$ depends only on the edges with an endpoint in~$[-M,M]^2$, in~$T$ and in~$T_N$~\citep[supplementary material, proof of Claim~1]{AthreyaJarai2006}, so the weak convergence of~$T_N$ to~$T$ gives
\begin{equation}\label{pre:past-limit}
\PP\bigl(\operatorname{Past}_{T_N}(x)\subset [-M,M]^2\bigr)\to\PP\bigl(\operatorname{Past}_T(x)\subset [-M,M]^2\bigr)\quad\text{as }N\to\infty\,.
\end{equation}

\subsection{Sandpiles}\label{ssec:prelim-sandpiles}

Let~$G$ be a finite connected multigraph with a sink and no loops, and define recurrent sandpiles on~$G$ as in Section~\ref{ssec:main-results}. A nonempty set~$B$ of non-sink vertices is a \textbf{forbidden subconfiguration} of a sandpile~$\eta$ if~$\eta(x)<\deg(x,B)$ for every~$x\in B$.

Let~$\eta$ be a stable sandpile on~$G$, and fix a set of non-sink vertices. \textbf{Burning}~$\eta$ from the sink while \textbf{holding} these vertices unburned proceeds in steps. At step zero, the sink burns. At each later step, every vertex~$x$ that is not held and is still unburned burns if~$\eta(x)$ is at least the number of edges from~$x$ to the vertices unburned at the start of the step. Eventually no vertex burns, and the vertices still unburned then form the \textbf{final unburned set}. When no vertex is held, burning is the burning algorithm of~\citet[page~1615]{Dhar1990}. In general, burning is the first phase of the burning in~\citet[Section~3.1]{JaraiWerning2014}.

The following facts go back to~\citet{Dhar1990} and~\citet{MajumdarDhar1992}. By~\citet[Lemma~4.2]{HolroydLevineMeszarosPeresProppWilson2008}, every stable sandpile~$\eta$ on~$G$ satisfies
\begin{equation}\label{pre:burning}
\eta\text{ is recurrent}\quad\text{exactly when}\quad\eta\text{ has no forbidden subconfiguration}\,.
\end{equation}
By~\citet[Corollary~4.4]{Jarai2018}, the burning algorithm gives a bijection between the recurrent sandpiles on~$G$ and the spanning trees of~$G$, so
\begin{equation}\label{pre:count}
\bigl|\{\text{recurrent sandpiles on }G\}\bigr|=\bigl|\operatorname{SpanT}(G)\bigr|\,.
\end{equation}
By~\citet[Theorem~2.11 and Section~2.3]{Jarai2018},
\begin{equation}\label{pre:permutation}
\text{adding a grain at a non-sink vertex and stabilizing permutes the recurrent sandpiles on~$G$}\,.
\end{equation}

The next lemma restates an observation in~\citet[proof of Theorem~3]{JaraiWerning2014}, which follows from~\citet*[Lemmas~3 and~4]{GamlinJarai2014}.

\begin{lemma}\label{fv:held-burning}
Let~$G$ be a finite connected multigraph with a sink and no loops. Burn a uniform recurrent sandpile on~$G$ from the sink while holding a set of non-sink vertices unburned. Then the final unburned set has the law of the union of the pasts of the held vertices in a uniform spanning tree of~$G$ oriented toward the sink.
\end{lemma}

Let~$\mu_N$ be the uniform measure on the recurrent sandpiles of~$[-N,N]^2$. By~\citet[Theorem~1]{AthreyaJarai2004},
\begin{equation}\label{pre:sandpile-limit}
\mu_N\to\mu\quad\text{weakly as }N\to\infty\,.
\end{equation}

\subsection{Wilson's algorithm with two root sets}\label{ssec:prelim-wilson}

Fix a nonempty set of roots, either in a finite connected multigraph or in~$\Z^2$ with finite complement, and an enumeration of the vertices. \textbf{Wilson's algorithm}~\citep{Wilson1996} builds a rooted spanning forest with these roots. The algorithm starts with the forest that consists of the roots and has no edges. While some vertex is not in the forest, an independent simple random walk runs from the first such vertex in the enumeration until the walk hits the forest, and the loop-erasure of the walk, obtained by erasing its loops in chronological order and called a \textbf{branch}, is added to the forest. By~\citet[Theorem~4.1]{LyonsPeres2016}, applied after identifying the roots to one vertex, for every enumeration
\begin{equation}\label{pre:wilson}
\text{the resulting forest is uniformly distributed over the rooted spanning forests with these roots}\,.
\end{equation}

Let~$A$ and~$\Gamma$ be disjoint nonempty sets of vertices, and run Wilson's algorithm with roots~$A\cup\Gamma$. A vertex is \textbf{joined to}~$\Gamma$ if its component in the resulting forest contains a vertex of~$\Gamma$, and \textbf{joined to}~$A$ otherwise. When~$A\cup\Gamma\subset\Z^2$ has finite complement, let~$\nu_{A,\Gamma}$ be the law of the set of vertices joined to~$\Gamma$. By~\eqref{pre:wilson}, the law~$\nu_{A,\Gamma}$ does not depend on the enumeration. Enumerating a vertex~$y$ first gives
\begin{equation}\label{fo:color-marginal}
\nu_{A,\Gamma}(\text{$y$ is joined to~$\Gamma$})=\PP_y(\text{the walk visits~$\Gamma$ before~$A$})\,.
\end{equation}
For a set~$V\subset\Z^2$, let~$\cL(V)$ be the number of lattice edges from~$V$ to vertices joined to~$\Gamma$. Thus, if the common ancestor of~$S$ in a tree~$t$ in~$[-N,N]^2$ is not~$s_N$ and the vertices outside its past are the vertices joined to~$\Gamma$, then the count~$\cL_t$ of Section~\ref{ssec:prelim-trees} equals~$\cL(\operatorname{MinSt}(t,S))$.

\begin{lemma}\label{fo:lattice-monotonicity}
Let~$A\subset A'$ and~$\Gamma'\subset\Gamma$ be subsets of~$\Z^2$ for which~$\nu_{A,\Gamma}$ and~$\nu_{A',\Gamma'}$ are defined. Then~$\nu_{A',\Gamma'}$ is stochastically dominated by~$\nu_{A,\Gamma}$. In particular, for every set~$V\subset\Z^2$, the law of~$\cL(V)$ under~$\nu_{A',\Gamma'}$ is stochastically dominated by its law under~$\nu_{A,\Gamma}$.
\end{lemma}

\begin{proof}
Pass from the root sets~$A$ and~$\Gamma$ to~$A'$ and~$\Gamma'$ in two steps: remove the roots in~$\Gamma\setminus(A'\cup\Gamma')$, then add the roots in~$A'\setminus(A\cup\Gamma)$ and count the roots in~$\Gamma\cap A'$ as roots of~$A'$. Adding roots identifies vertices, so by~\citet[Exercise~10.8 and Theorem~10.4]{LyonsPeres2016} each step has a coupling in which the forest with more roots is contained in the forest with fewer roots. In these couplings, every vertex joined to~$A$ at the start is still joined to~$A$ after the first step, which keeps the roots of~$A$, and is joined to~$A'$ after the second step, which adds only roots of~$A'\setminus A$.
\end{proof}

\subsection{Electrical networks and random walk}\label{ssec:prelim-electrical}

We view a multigraph as an electrical network in which every edge has conductance one. The \textbf{energy} of a function~$f$ on the vertices is~$\sum_{xy}(f(x)-f(y))^2$, where the sum runs over the edges~$xy$. Let~$A$ and~$\Gamma$ be disjoint nonempty sets of vertices, either of a finite connected network or of~$\Z^2$ with one of them finite and the other with finite complement. In the second case we work in the finite network obtained from~$\Z^2$ by identifying~$A$ to one vertex and~$\Gamma$ to another and deleting the resulting loops, and we view functions on this network as functions on~$\Z^2$. The \textbf{voltage} between~$A$ and~$\Gamma$ is the function that equals one on~$A$ and zero on~$\Gamma$ and is harmonic at every other vertex. The \textbf{current} through an edge from~$x$ to~$y$ is the voltage at~$x$ minus the voltage at~$y$. The \textbf{effective conductance}~$\mathcal C(A\leftrightarrow\Gamma)$ between~$A$ and~$\Gamma$ is the energy of the voltage, and the \textbf{effective resistance} is~$1/\mathcal C(A\leftrightarrow\Gamma)$. By~\citet[equation~(2.2) and the uniqueness principle in Section~2.1]{LyonsPeres2016}, every vertex~$y$ satisfies
\begin{equation}\label{pre:voltage}
(\text{the voltage between~$A$ and~$\Gamma$ at~$y$})=\PP_y(\text{the walk visits~$A$ before~$\Gamma$})\,.
\end{equation}
By Dirichlet's principle~\citep[Exercises~2.13 and~2.79]{LyonsPeres2016},
\begin{equation}\label{pre:dirichlet}
\mathcal C(A\leftrightarrow\Gamma)=\min_f\sum_{xy}(f(x)-f(y))^2\,,
\end{equation}
where~$f$ runs over the functions that equal one on~$A$ and zero on~$\Gamma$. By~\citet[Exercise~2.47 and equation~(2.7)]{LyonsPeres2016},
\begin{equation}\label{pre:count-conductance}
\sum_{xy:\,x\in A}\PP_y(\text{the walk visits~$\Gamma$ before~$A$})=\mathcal C(A\leftrightarrow\Gamma)\,.
\end{equation}
\section{Reduction to the uniform spanning tree}\label{rd:section}

Fix an integer~$r\geq2$, and let~$S$ consist of the origin, its four neighbors, and~$(r,0)$. In this section we deduce Theorem~\ref{th:main} from a reduction to the uniform spanning tree and two results about uniform spanning trees and forests. Proposition~\ref{sp:reciprocal-reduction}, which we prove in Section~\ref{sp:section}, bounds the tail of the radius by~$\limsup_{N\to\infty}\EE[\cL_N^{-1}]$ plus~$Cr(M-r)^{-\nf13}$, where~$\cL_N$ is the count~$\cL_t$ of Section~\ref{ssec:prelim-trees} for~$t=T_N$. Lemma~\ref{rd:conditional-law}, which we prove here, and Proposition~\ref{ct:reciprocal}, which we prove in Sections~\ref{fo:section}--\ref{sec:electrical}, bound~$\EE[\cL_N^{-1}]$.

\begin{proposition}\label{sp:reciprocal-reduction}
There is a universal constant~$C<\infty$ such that every integer~$M\geq r+2$ satisfies
\begin{equation}\label{sp:reduction-bound}
\mu(\rho>M)\leq\limsup_{N\to\infty}\EE\bigl[\cL_N^{-1}\bigr]+Cr(M-r)^{-\nf13}\,.
\end{equation}
\end{proposition}

Let~$K$ be a finite connected set of lattice vertices, and let~$\gamma\subset\Z^2\setminus K$ be a connected set with finite complement that contains a neighbor of a vertex of~$K$. Under~$\nu_{K,\gamma}$, the count~$\cL(K)$ of Section~\ref{ssec:prelim-wilson} satisfies~$\cL(K)\geq1$, since~$\gamma$ contains a neighbor of~$K$.

\begin{proposition}\label{ct:reciprocal}
There are universal constants~$a>0$ and~$C<\infty$ such that every~$K$ and~$\gamma$ as above satisfy
\begin{equation*}
\EE_{\nu_{K,\gamma}}\bigl[\cL(K)^{-1}\bigr]\leq C\bigl(\log(\diam_\infty K+2)\bigr)^{-a}\,.
\end{equation*}
\end{proposition}

The next lemma links the two propositions.

\begin{lemma}\label{rd:conditional-law}
Let~$N\geq r$. On the event that the common ancestor of~$S$ in~$T_N$ is not~$s_N$, let~$\Gamma_N$ be the union of~$\Z^2\setminus [-N,N]^2$ and the set of vertices of~$[-N,N]^2$ that follow the common ancestor on the path in~$T_N$ from the common ancestor to~$s_N$. Then, on this event, the set~$\Gamma_N$ avoids~$\operatorname{MinSt}(T_N,S)$, and the conditional law of~$\cL_N$ given~$\operatorname{MinSt}(T_N,S)$ and this path is the law of~$\cL(\operatorname{MinSt}(T_N,S))$ under~$\nu_{\operatorname{MinSt}(T_N,S),\Gamma_N}$.
\end{lemma}

\begin{proof}
By~\eqref{pre:wilson}, we may generate~$T_N$ by Wilson's algorithm on~$[-N,N]^2$ with root~$s_N$, enumerating the vertices of~$S$ first. The branches from~$S$ form~$\operatorname{MinSt}(T_N,S)$ and the path in~$T_N$ from the common ancestor of~$S$ to~$s_N$. Given these branches, the later walks carry out Wilson's algorithm with the vertices of these branches as roots. Suppose that the common ancestor is not~$s_N$. Then~$\Gamma_N$ avoids the past of the common ancestor, which contains~$\operatorname{MinSt}(T_N,S)$. By~\eqref{pre:wilson}, with~$s_N$ identified with~$\Z^2\setminus[-N,N]^2$, which lies in~$\Gamma_N$, the conditional law of the set of vertices of~$[-N,N]^2$ joined to~$\Gamma_N$ given these branches is its law under~$\nu_{\operatorname{MinSt}(T_N,S),\Gamma_N}$. A vertex is joined to~$\operatorname{MinSt}(T_N,S)$ exactly when the oriented path from the vertex passes through the common ancestor, so~$\cL_N=\cL(\operatorname{MinSt}(T_N,S))$.
\end{proof}

The proof of Theorem~\ref{th:main} combines these results.

\begin{proof}[Proof of Theorem~\ref{th:main} assuming Propositions~\ref{sp:reciprocal-reduction} and~\ref{ct:reciprocal}]
Let~$N\geq r$, and suppose that the common ancestor of~$S$ in~$T_N$ is not~$s_N$. Then~$K=\operatorname{MinSt}(T_N,S)$ and~$\gamma=\Gamma_N$ satisfy the assumptions of Proposition~\ref{ct:reciprocal}, with~$\Gamma_N$ disjoint from~$K$ by Lemma~\ref{rd:conditional-law}, and~$\diam_\infty K\geq r$ since~$K$ contains~$0$ and~$(r,0)$. Since also~$\cL_N^{-1}=0$ when the common ancestor is~$s_N$, Lemma~\ref{rd:conditional-law} and Proposition~\ref{ct:reciprocal} give~$\EE[\cL_N^{-1}]\leq C(\log(r+2))^{-a}$ for every~$N\geq r$. For an integer~$M\geq64$, taking~$r=\lfloor M^{\nf16}\rfloor$ in~\eqref{sp:reduction-bound} gives
\begin{equation*}
\mu(\rho>M)\leq C(\log(r+2))^{-a}+Cr(M-r)^{-\nf13}\leq C(\log(M+2))^{-a}\,.
\end{equation*}
Increasing~$C$ extends this bound to all real~$M\geq0$, and letting~$M$ tend to infinity shows that~$\eta+\delta_0$ stabilizes for $\mu$-almost every~$\eta$.
\end{proof}

\section{From avalanches to spanning trees}\label{sp:section}

In this section we prove Proposition~\ref{sp:reciprocal-reduction}. Fix an integer~$M\geq r+2$, with~$r$ and~$S$ as in Section~\ref{rd:section}. We first work in a box~$[-N,N]^2$ with~$N>M+1$, under the uniform measure~$\mu_N$ on recurrent sandpiles of Section~\ref{ssec:prelim-sandpiles}, and for a recurrent sandpile~$\eta'$ on~$[-N,N]^2$ we write~$\rho_N$ for the radius of the avalanche of~$\eta'+\delta_0$ on~$[-N,N]^2$. Section~\ref{ssec:waves} bounds~$\mu_N(\rho_N>M)$ by the numbers of two kinds of waves: a wave that contains~$S$ followed by a wave that misses a vertex of~$S$, and a wave that misses a vertex of~$S$ and is not contained in~$[-M,M]^2$. Section~\ref{ssec:waves-to-trees} bounds the second number by the probabilities that the pasts in~$T_N$ of~$r+4$ vertices of~$[-r,r]^2$ leave~$[-M,M]^2$, and the first by~$\EE[\cL_N^{-1}]$. Section~\ref{sec:limits} lets~$N\to\infty$.

\subsection{Waves}\label{ssec:waves}

Let~$\mathcal{R}^{(0)}_N$ be the set of recurrent sandpiles on~$[-N,N]^2$ with~$0$ as an additional sink. We extend each~$\eta\in\mathcal{R}^{(0)}_N$ by~$\eta(0)\coloneqq-1$ to a sandpile on~$[-N,N]^2$, in which the origin is not a sink. By~\eqref{pre:burning}, the extended~$\eta$ has no forbidden subconfiguration avoiding~$0$, and~$\{0\}$ is a forbidden subconfiguration. Let~$W(\eta)$ be the union of all forbidden subconfigurations of the extended~$\eta$. A union of forbidden subconfigurations is a forbidden subconfiguration, so~$W(\eta)$ is the largest one and contains~$0$.

As in~\citet[equations~(7.1) and~(7.3)]{JaraiRedig2008}, let~$U$ act on~$\mathcal{R}^{(0)}_N$ by adding one grain to each neighbor of~$0$ and stabilizing with both~$0$ and~$s_N$ as sinks. By~\eqref{pre:permutation}, the operator~$U$ is a permutation of~$\mathcal{R}^{(0)}_N$. The next lemma makes precise an observation of~\citet*[Section~2]{IvashkevichKtitarevPriezzhev1994}.

\begin{lemma}\label{fv:inverse-waves}
Let~$\eta\in \mathcal{R}^{(0)}_N$. Starting from the sandpile with height four at the origin and~$\eta$ elsewhere, topple the origin once and then stabilize the other vertices with~$0$ as a sink. Then exactly the vertices of~$W(U\eta)$ topple, each once.
\end{lemma}

\begin{proof}
Toppling the origin once adds one grain to each neighbor of~$0$, so stabilizing the other vertices with~$0$ as a sink then leaves~$U\eta$ away from the origin. The vertices that topple form a wave, so each vertex other than the origin topples at most once~\citep*[Section~2.2]{BhupatirajuHansonJarai2017}. Let~$\Omega$ be the set of vertices that topple, including~$0$. Every~$x\in \Omega\setminus\{0\}$ satisfies~$(U\eta)(x)=\eta(x)-4+\deg(x,\Omega)<\deg(x,\Omega)$, since~$\eta(x)\leq3$. Hence~$\Omega$ is a forbidden subconfiguration of the extended~$U\eta$, and~$\Omega\subset W(U\eta)$. Every~$x\in W(U\eta)\setminus \Omega$ satisfies
\begin{equation*}
 \eta(x)+\deg(x,\Omega)=(U\eta)(x)<\deg\bigl(x,W(U\eta)\bigr)\leq\deg\bigl(x,W(U\eta)\setminus \Omega\bigr)+\deg(x,\Omega)\,,
\end{equation*}
so~$W(U\eta)\setminus \Omega$ is empty, since otherwise this set would be a forbidden subconfiguration of~$\eta$ avoiding~$0$. Thus~$\Omega=W(U\eta)$.
\end{proof}

By~\eqref{pre:burning}, the sandpiles that are recurrent on~$[-N,N]^2$ with an extra edge from~$0$ to~$s_N$, but not on~$[-N,N]^2$, are those with height four at the origin that lie in~$\mathcal{R}^{(0)}_N$ away from the origin. Hence, by~\citet[Definition~1 and Lemma~2.7]{BhupatirajuHansonJarai2017}, for every recurrent sandpile~$\eta'$ on~$[-N,N]^2$ with~$\eta'(0)=3$, the sandpile just before each wave of the avalanche of~$\eta'+\delta_0$ has height four at the origin and equals some~$\eta\in \mathcal{R}^{(0)}_N$ away from the origin. By Lemma~\ref{fv:inverse-waves}, this wave topples exactly~$W(U\eta)$, and the preceding wave, if there is one, leaves~$\eta$ away from the origin and therefore topples exactly~$W(\eta)$. Every~$\eta\in \mathcal{R}^{(0)}_N$ arises in this way from exactly one wave of exactly one of these avalanches.

Let~$\eta'$ be a recurrent sandpile on~$[-N,N]^2$ such that the avalanche of~$\eta'+\delta_0$ has radius~$\rho_N>M$. Then~$\eta'(0)=3$, since otherwise no vertex topples. The last wave misses a neighbor of~$0$~\citep[Corollary~2.8]{BhupatirajuHansonJarai2017} and therefore does not contain~$S$. Since~$\rho_N>M$, some wave is not contained in~$[-M,M]^2$. Hence either some wave neither contains~$S$ nor lies in~$[-M,M]^2$, or a wave not containing~$S$ immediately follows a wave containing~$S$. Let~$\eta$ be the sandpile away from the origin just before such a wave in the first case, or just before the later wave in the second case. Then~$S\not\subset W(U\eta)$, and either~$W(U\eta)\not\subset [-M,M]^2$ or~$S\subset W(\eta)$. By the previous paragraph, distinct~$\eta'$ give distinct~$\eta$. Since~$\mu_N$ gives mass~$\lvert\operatorname{SpanT}([-N,N]^2)\rvert^{-1}$ to each recurrent sandpile by~\eqref{pre:count}, and since~$U$ permutes~$\mathcal{R}^{(0)}_N$, we obtain
\begin{equation}\label{fv:excursion-union}
\begin{split}
\mu_N(\rho_N>M)&\leq\lvert\operatorname{SpanT}([-N,N]^2)\rvert^{-1}\bigl|\{\eta\in \mathcal{R}^{(0)}_N:S\subset W(\eta),\ S\not\subset W(U\eta)\}\bigr|\\
&\quad+\lvert\operatorname{SpanT}([-N,N]^2)\rvert^{-1}\bigl|\{\eta\in \mathcal{R}^{(0)}_N:S\not\subset W(\eta),\ W(\eta)\not\subset [-M,M]^2\}\bigr|\,.
\end{split}
\end{equation}

\subsection{From waves to spanning trees}\label{ssec:waves-to-trees}

We bound the two terms of~\eqref{fv:excursion-union} by spanning-tree counts. Fix a set~$\Omega$ with~$0\in \Omega\subset [-N,N]^2$, and let~$\Omega^c$ be the complement of~$\Omega$ in the vertex set of the graph~$[-N,N]^2$, so that~$s_N\in \Omega^c$. We also write~$\Omega$ and~$\Omega^c$ for the induced subgraphs of~$[-N,N]^2$, with sink~$0$ for~$\Omega$ and sink~$s_N$ for~$\Omega^c$.

As in~\citet[proof of Theorem~3]{JaraiWerning2014}, restricting~$\eta$ to~$\Omega\setminus\{0\}$ and taking~$\eta-\deg(\cdot,\Omega)$ on~$[-N,N]^2\setminus \Omega$ maps the sandpiles~$\eta\in \mathcal{R}^{(0)}_N$ with~$W(\eta)=\Omega$ injectively to pairs of sandpiles on~$\Omega$ and on~$\Omega^c$ that are stable and have no forbidden subconfiguration, because~$\eta$ has no forbidden subconfiguration avoiding the origin, $\Omega$ is the largest forbidden subconfiguration of~$\eta$, and~$\eta\leq3$.

By~\eqref{pre:burning} and~\eqref{pre:count}, on each of the graphs~$\Omega$ and~$\Omega^c$ the number of stable sandpiles with no forbidden subconfiguration equals the number of spanning trees, and both numbers vanish when the graph is disconnected. Counting the pairs therefore gives
\begin{equation}\label{fv:shape-mass}
 \bigl|\{\eta\in \mathcal{R}^{(0)}_N:W(\eta)=\Omega\}\bigr|\leq\lvert\operatorname{SpanT}(\Omega)\rvert\cdot\lvert\operatorname{SpanT}(\Omega^c)\rvert\,,
\end{equation}
and~$\lvert\operatorname{SpanT}(\Omega)\rvert\cdot\lvert\operatorname{SpanT}(\Omega^c)\rvert$ is the number of spanning forests of~$[-N,N]^2$ with two components, one with vertex set~$\Omega$ and the other containing~$s_N$.

Let~$x$ and~$y$ be adjacent vertices of~$[-N,N]^2$. Adding the edge~$xy$ maps the spanning forests of~$[-N,N]^2$ with two components, one containing~$0$ and~$x$ and the other containing~$s_N$ and~$y$, injectively to spanning trees of~$[-N,N]^2$ in which~$xy$ is the outgoing edge of~$x$ and the past of~$x$ is the component of~$0$. Every set that contains~$0$ and misses a vertex of~$S$ contains a vertex~$x$ of the union of~$S$ and the segment from~$0$ to~$(r,0)$ that has a neighbor~$y$ in this union outside the set. Summing~\eqref{fv:shape-mass} over the sets~$\Omega$ that contain~$0$, miss a vertex of~$S$, and are not contained in~$[-M,M]^2$ therefore gives
\begin{equation}\label{fv:large-omission}
 \bigl|\{\eta\in \mathcal{R}^{(0)}_N:S\not\subset W(\eta),\ W(\eta)\not\subset [-M,M]^2\}\bigr|
 \leq\lvert\operatorname{SpanT}([-N,N]^2)\rvert\sum_x\PP\bigl(\operatorname{Past}_{T_N}(x)\not\subset [-M,M]^2\bigr)\,,
\end{equation}
where~$x$ runs over the~$r+4$ vertices of the union of~$S$ and the segment from~$0$ to~$(r,0)$.

The proof of the next lemma makes precise an observation of~\citet[page~2095]{PriezzhevKtitarevIvashkevich1996} and~\citet[page~2]{PaczuskiBoettcher1997}: each wave repeats the topplings of the preceding wave up to the boundary of the preceding wave.

\begin{lemma}\label{fv:outflow}
We have
\begin{equation*}
 \bigl|\{\eta\in \mathcal{R}^{(0)}_N:S\subset W(\eta),\ S\not\subset W(U\eta)\}\bigr|\leq\lvert\operatorname{SpanT}([-N,N]^2)\rvert\,\EE\bigl[\cL_N^{-1}\bigr]\,.
\end{equation*}
\end{lemma}

\begin{proof}
Fix a set~$\Omega$ with~$S\subset \Omega\subset [-N,N]^2$, and let~$\eta\in \mathcal{R}^{(0)}_N$ satisfy~$W(\eta)=\Omega$. As in Section~\ref{ssec:prelim-sandpiles}, burn the restriction of~$\eta$ to~$\Omega\setminus\{0\}$, as a sandpile on~$\Omega$ with sink~$0$, while holding unburned every vertex of~$\Omega$ with a neighbor in~$\Omega^c$. The four neighbors of the origin lie in~$S\subset \Omega$, so the origin has no neighbor in~$\Omega^c$ and is not held. A burned vertex~$x\neq0$ is not held, so its four neighbors lie in~$\Omega$, and by the burning rule~$x$ has at least four grains once its neighbors burned at earlier steps have toppled. Hence, as in~\citet*[Section~7]{JaraiRedig2008}, after the origin topples in Lemma~\ref{fv:inverse-waves}, the other burned vertices can topple legally in the order of their steps. By the abelian property, these topplings are part of the stabilization in Lemma~\ref{fv:inverse-waves}, so every vertex of~$\Omega$ outside the final unburned set lies in~$W(U\eta)$.

The final unburned set depends only on the restriction of~$\eta$ to~$\Omega\setminus\{0\}$, which is a recurrent sandpile on~$\Omega$ by the injection above and~\eqref{pre:burning}. By Lemma~\ref{fv:held-burning} for~$\Omega$ with sink~$0$, the final unburned set of a uniform recurrent sandpile on~$\Omega$ has the law of the union of the pasts of the held vertices in a uniform spanning tree of~$\Omega$ oriented toward~$0$. In a spanning tree~$t$ of~$\Omega$ oriented toward~$0$, the count~$\cL_t$ of Section~\ref{ssec:prelim-trees} is the number of edges of~$[-N,N]^2$ from~$\operatorname{MinSt}(t,S)$ to~$\Omega^c$. Since~$\operatorname{MinSt}(t,S)$ is the union of the oriented paths in~$t$ from the vertices of~$S$ to~$0$, the union of the pasts of the held vertices in~$t$ meets~$S$ exactly when~$\cL_t>0$. Hence, by~\eqref{pre:count}, the number of recurrent sandpiles on~$\Omega$ whose final unburned set meets~$S$ is the number of spanning trees~$t$ of~$\Omega$ oriented toward~$0$ with~$\cL_t>0$. If~$S\not\subset W(U\eta)$, then the final unburned set meets~$S$, because every vertex of~$\Omega$ outside the final unburned set lies in~$W(U\eta)$. The injection into pairs of sandpiles on~$\Omega$ and on~$\Omega^c$ therefore gives
\begin{equation}\label{fv:shape-outflow}
 \bigl|\{\eta\in \mathcal{R}^{(0)}_N:W(\eta)=\Omega,\ S\not\subset W(U\eta)\}\bigr|\leq\lvert\operatorname{SpanT}(\Omega^c)\rvert\cdot\bigl|\{t:\cL_t>0\}\bigr|\,,
\end{equation}
where~$t$ runs over the spanning trees of~$\Omega$ oriented toward~$0$.

Fix a spanning tree~$t$ of~$\Omega$ oriented toward~$0$ with~$\cL_t>0$, and, as in Figure~\ref{fig:joining} and in~\citet[proof of Lemma~5]{LevinePeres2014}, join~$t$ to a spanning tree of~$\Omega^c$ by one of the~$\cL_t$ edges of~$[-N,N]^2$ from a vertex~$x\in \operatorname{MinSt}(t,S)$ to~$\Omega^c$. In the joined tree, oriented toward~$s_N$, the common ancestor of~$S$ is~$x$, the outgoing edge of~$x$ is the edge used to join~$t$, the past of~$x$ is~$\Omega$, and the minimal subtree containing~$S$ is~$\operatorname{MinSt}(t,S)$. So the joined tree has count~$\cL_t$, and weighting each joined tree by~$\cL_t^{-1}$ counts each pair of trees on~$\Omega$ and~$\Omega^c$ once. Deleting the outgoing edge of the common ancestor of~$S$ recovers~$\Omega$, the tree~$t$, the edge used to join~$t$, and the spanning tree of~$\Omega^c$, so each spanning tree of~$[-N,N]^2$ arises at most once. Summing~\eqref{fv:shape-outflow} over the sets~$\Omega$ with~$S\subset \Omega\subset [-N,N]^2$, we obtain
\begin{equation*}
\begin{split}
 \bigl|\{\eta\in \mathcal{R}^{(0)}_N:S\subset W(\eta),\ S\not\subset W(U\eta)\}\bigr|&\leq\sum_{S\subset \Omega\subset [-N,N]^2}\lvert\operatorname{SpanT}(\Omega^c)\rvert\cdot\bigl|\{t:\cL_t>0\}\bigr|\\
 &\leq\lvert\operatorname{SpanT}([-N,N]^2)\rvert\,\EE\bigl[\cL_N^{-1}\bigr]\,.
\end{split}
\end{equation*}
\end{proof}

\subsection{Passing to the infinite lattice}\label{sec:limits}

To let~$N\to\infty$, we bound the tail of the past of a vertex in~$T$.

\begin{lemma}\label{rd:past-tail}
There is a universal constant~$C<\infty$ such that~$\PP\bigl(\operatorname{Past}_T(x)\not\subset [-M,M]^2\bigr)\leq C(M-r)^{-\nf13}$ for every~$x\in [-r,r]^2$.
\end{lemma}

\begin{proof}
By~\citet[Proposition~5.6 and Theorem~12.2]{BLPS2001}, the duals of the lattice edges absent from~$T$ form a uniform spanning tree of the dual lattice. Since~$\operatorname{Past}_T(x)$ and its complement are connected, the duals of the edges between them form a simple cycle, which consists of the dual of the outgoing edge of~$x$ and the path in the dual tree between the endpoints of this dual edge. If~$\operatorname{Past}_T(x)\not\subset [-M,M]^2$, then~$\operatorname{Past}_T(x)$ contains a vertex at~$|\cdot|_\infty$-distance greater than~$M-r$ from~$x$, and since the cycle surrounds~$\operatorname{Past}_T(x)$, this path reaches Euclidean distance greater than~$M-r$ from both endpoints. \citet[Lemma~2.6]{BarlowMasson2011} bounded the probability that the path in the dual tree between two adjacent dual vertices leaves the closed Euclidean ball of radius~$M-r$ about one of them by~$C(M-r)^{-\nf13}$. Summing this bound over the four edges at~$x$ proves the lemma.
\end{proof}

The proof of Proposition~\ref{sp:reciprocal-reduction} combines Lemma~\ref{rd:past-tail} with the bounds of Sections~\ref{ssec:waves} and~\ref{ssec:waves-to-trees}.

\begin{proof}[Proof of Proposition~\ref{sp:reciprocal-reduction}]
Since~$\{\rho\leq M\}$ on~$\Z^2$ and~$\{\rho_N\leq M\}$ on~$[-N,N]^2$ are determined by the heights in~$[-M-1,M+1]^2$ in the same way, \eqref{pre:sandpile-limit} gives
\begin{equation*}
\mu(\rho>M)=\lim_{N\to\infty}\mu_N(\rho_N>M)\,.
\end{equation*}
For~$N>M+1$, the bound~\eqref{fv:excursion-union}, Lemma~\ref{fv:outflow}, and the bound~\eqref{fv:large-omission} give
\begin{equation*}
\mu_N(\rho_N>M)\leq\EE\bigl[\cL_N^{-1}\bigr]+\sum_x\PP\bigl(\operatorname{Past}_{T_N}(x)\not\subset [-M,M]^2\bigr)\,,
\end{equation*}
with~$x$ as in~\eqref{fv:large-omission}. By~\eqref{pre:past-limit} and Lemma~\ref{rd:past-tail}, each of the~$r+4\leq3r$ terms of the sum has limit at most~$C(M-r)^{-\nf13}$ as~$N\to\infty$, so taking the limit superior of the previous display proves~\eqref{sp:reduction-bound}.
\end{proof}

\section{A lower bound at one scale}\label{fo:section}

In this section we prove Proposition~\ref{fo:local-comparison} below by reducing it to a variance bound in a finite planar network, Proposition~\ref{el:reference-moment}, which we prove in Section~\ref{sec:electrical}.

\begin{proposition}\label{fo:local-comparison}
There are universal constants~$\varepsilon>0$ and~$\ell_*<\infty$ such that the following holds. Let~$\ell\geq\ell_*$ be an integer, and let~$A$ and~$\Gamma$ be disjoint sets of vertices of~$\Z^2$ whose union has finite complement, such that~$0\in A$ and~$\Gamma$ contains a neighbor of~$0$. Assume that~$A$ is connected and meets~$\Z^2\setminus [-2\ell+1,2\ell-1]^2$, and that~$\Gamma$ is connected and meets~$\partial[-\ell,\ell]^2$. Then
\begin{equation*}
\nu_{A,\Gamma}\bigl(\cL(A\cap [-2\ell,2\ell]^2)\geq\tfrac18\log \ell\bigr)\geq\varepsilon\,.
\end{equation*}
\end{proposition}

\begin{figure}[htbp]
\centering
\includegraphics[width=\textwidth]{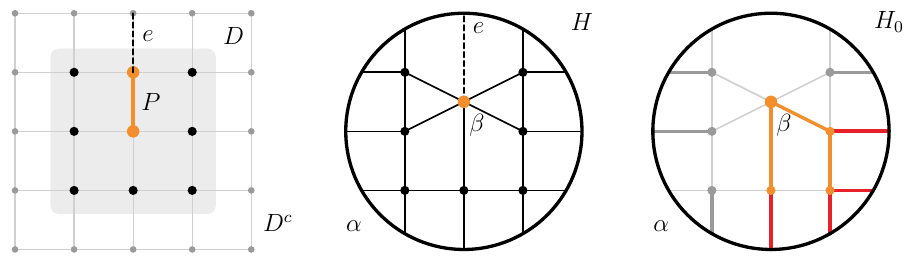}
\caption{For the sets~$D$ and~$P$ and the edge~$e$ defined before Proposition~\ref{el:reference-moment}, identifying~$D^c$ to~$\alpha$ and~$P$ to~$\beta$, and deleting the resulting loops, gives~$H$, and deleting~$e$ from~$H$ gives~$H_0$. In~$H_0$, the thick gray and orange edges form a sample of the uniform rooted spanning forest, and for this sample~$\cL_{H_0}$ is the number of red edges.}
\label{fig:network}
\end{figure}

Let~$D\subset\Z^2$ be finite and connected with connected complement~$D^c$, let~$P\subset D$ be nonempty and connected, and let~$e$ be a lattice edge from~$D^c$ to~$P$. Let~$H$ be the planar multigraph obtained by identifying~$D^c$ to a vertex~$\alpha$ and~$P$ to a vertex~$\beta$ and deleting the resulting loops, and let~$H_0$ be~$H$ with~$e$ deleted, as in Figure~\ref{fig:network}. Let~$\cL_{H_0}$ be the number of edges of~$H_0$ from~$\alpha$ to vertices joined to~$\beta$ in the uniform rooted spanning forest of~$H_0$ with roots~$\alpha$ and~$\beta$. Since the endpoints of~$e$ are the two roots, the graphs~$H$ and~$H_0$ have the same rooted spanning forests with roots~$\alpha$ and~$\beta$, and by~\eqref{pre:wilson} the vertices of~$D$ joined to~$P$ under~$\nu_{D^c,P}$ are those joined to~$\beta$ in the uniform rooted spanning forest of~$H_0$. Let~$\phi$ be the voltage between~$\beta$ and~$\alpha$ in~$H_0$, which we extend to~$\Z^2$ by zero on~$D^c$ and one on~$P$, and let~$m$ be the number of edges of~$H_0$ at~$\alpha$.

\begin{proposition}\label{el:reference-moment}
For all sets~$D$ and~$P$ and every edge~$e$ as above, the mean~$\EE[\cL_{H_0}]$ is the energy of~$\phi$ in~$H_0$, and, for a universal constant~$C<\infty$,
\begin{equation}\label{el:reference-estimate}
\operatorname{Var}(\cL_{H_0})\leq C\,\EE[\cL_{H_0}]\Bigl(1+\log\frac m{\EE[\cL_{H_0}]}\Bigr)\,.
\end{equation}
\end{proposition}

\begin{proof}[Proof of Proposition~\ref{fo:local-comparison} assuming Proposition~\ref{el:reference-moment}]
\emph{Step 1.} We find sets~$D$ and~$P$ and an edge~$e$ as above such that~$\nu_{D^c,P}$ is stochastically dominated by~$\nu_{A,\Gamma}$, and such that under~$\nu_{D^c,P}$
\begin{equation}\label{fo:count-split}
\cL_{H_0}=\cL(A)-1+\cL(D^c\setminus A)\,.
\end{equation}
By the assumptions on~$\Gamma$, there is a nearest-neighbor path in~$\Gamma\cap [-\ell,\ell]^2$ from a neighbor of~$0$ to~$\partial[-\ell,\ell]^2$. Let~$P$ be the vertex set of this path, let~$e$ be the edge from~$0$ to the first vertex of the path, and let~$D$ be the component of~$[-2\ell+1,2\ell-1]^2\setminus A$ that contains~$P$. Observe that the sets~$D$ and~$P$ and the edge~$e$ satisfy the assumptions in the paragraph before Proposition~\ref{el:reference-moment}. Since~$A\subset D^c$ and~$P\subset\Gamma$, Lemma~\ref{fo:lattice-monotonicity} shows that~$\nu_{D^c,P}$ is stochastically dominated by~$\nu_{A,\Gamma}$. The edges of~$H_0$ at~$\alpha$ are the lattice edges from~$D^c$ to~$D$ other than~$e$, so~$\cL_{H_0}=\cL(D^c)-1$ under~$\nu_{D^c,P}$, which is~\eqref{fo:count-split}.

\emph{Step 2.} We bound the first two moments of~$\cL_{H_0}$: with~$C$ the constant of Proposition~\ref{el:reference-moment}, we show that
\begin{equation}\label{fo:moments}
\EE[\cL_{H_0}]\geq\tfrac12\log \ell\geq32\qquad\text{and}\qquad\EE[\cL_{H_0}^2]\leq(1+8C)(\EE[\cL_{H_0}])^2\,.
\end{equation}
By Proposition~\ref{el:reference-moment}, the mean~$\EE[\cL_{H_0}]$ is the energy of~$\phi$ in~$H_0$, which is the sum of~$(\phi(x)-\phi(y))^2$ over the lattice edges~$xy$ other than~$e$. For~$1\leq j\leq\ell$, the cycle of~$8j$ lattice edges along~$\partial[-j,j]^2$ avoids~$e$ and meets~$A$, where~$\phi$ vanishes, and~$P$, where~$\phi$ equals one, so by the Cauchy--Schwarz inequality it contributes at least~$1/(2j)$ to the energy of~$\phi$. Summing over~$j$ gives~$\EE[\cL_{H_0}]\geq\tfrac12\log\ell$. Let~$\ell_*\coloneqq10^{30}$, so that~$\tfrac12\log\ell\geq32$ for~$\ell\geq\ell_*$. Since~$m\leq\ell^3$, the first bound in~\eqref{fo:moments} gives~$1+\log(m/\EE[\cL_{H_0}])\leq8\,\EE[\cL_{H_0}]$, and~\eqref{el:reference-estimate} gives the second bound.

\emph{Step 3.} We bound the mean of the last term in~\eqref{fo:count-split}:
\begin{equation}\label{fo:loss}
\EE_{\nu_{D^c,P}}\bigl[\cL(D^c\setminus A)\bigr]\leq12\,.
\end{equation}
Since~$D$ is a component of~$[-2\ell+1,2\ell-1]^2\setminus A$, every edge from~$D^c\setminus A$ to~$D$ starts outside~$[-2\ell+1,2\ell-1]^2$. By~\eqref{fo:color-marginal}, the endpoint~$y\in D$ of such an edge is joined to~$P$ under~$\nu_{D^c,P}$ with probability at most~$\PP_y(\text{the walk visits~$[-\ell,\ell]^2$ before~$\Z^2\setminus [-2\ell+1,2\ell-1]^2$})$. By~\eqref{pre:count-conductance}, the left side of~\eqref{fo:loss} is therefore at most~$\mathcal C([-\ell,\ell]^2\leftrightarrow\Z^2\setminus[-2\ell+1,2\ell-1]^2)$. Applying~\eqref{pre:dirichlet} to the function~$x\mapsto\min\{1,\max\{0,2-|x|_\infty/\ell\}\}$, which changes by~$1/\ell$ along~$12\ell^2$ edges and is constant along the others, shows that this conductance is at most~$12$, which proves~\eqref{fo:loss}.

\emph{Step 4.} We apply the Paley--Zygmund inequality to~$\cL(A)-1$ and deduce the proposition: with~$\varepsilon\coloneqq1/(16(1+8C))$ for the~$C$ of Step~2, we show that
\begin{equation}\label{fo:pz}
\nu_{D^c,P}\bigl(\cL(A)-1\geq\EE[\cL_{H_0}]/4\bigr)\geq\varepsilon\,.
\end{equation}
Since~$e$ goes from~$0\in A$ to a vertex of~$P$, the identity~\eqref{fo:count-split} shows that~$\cL(A)-1$ lies between zero and~$\cL_{H_0}$ under~$\nu_{D^c,P}$. By~\eqref{fo:count-split}, \eqref{fo:moments}, and~\eqref{fo:loss}, the mean of~$\cL(A)-1$ is at least~$\EE[\cL_{H_0}]-12\geq\EE[\cL_{H_0}]/2$, and the second moment of~$\cL(A)-1$ is at most~$(1+8C)(\EE[\cL_{H_0}])^2$. The Paley--Zygmund inequality therefore gives~\eqref{fo:pz}.

Since every vertex joined to~$P$ lies in~$D\subset[-2\ell+1,2\ell-1]^2$ and~$\EE[\cL_{H_0}]/4\geq\tfrac18\log\ell$, the bound~\eqref{fo:pz} and Step~1 prove the proposition.
\end{proof}

\section{Repeating the estimate at many scales}\label{ct:section}

In this section we prove Proposition~\ref{ct:reciprocal}. We build the uniform rooted spanning forest with roots~$K\cup\gamma$ with Wilson's algorithm one scale at a time, at scales~$\ell_1<\ell_2<\cdots$, each a fixed power of the previous one, and at each scale we apply Proposition~\ref{fo:local-comparison} given the branches built at the earlier scales. Section~\ref{ssec:added-roots} extends Proposition~\ref{fo:local-comparison} to root sets that contain these branches, and Section~\ref{ssec:scale-by-scale} carries out the construction.

\subsection{Roots added inside a box}\label{ssec:added-roots}

At the scale~$\ell=\ell_i$, the branches built at the earlier scales become roots, so Wilson's algorithm continues with roots~$A\cup\Gamma$, where~$A\supset K$ adds the branches joined to~$K$ and~$\Gamma\supset\gamma$ adds the others. Lemma~\ref{ct:revealed-roots} below shows that~$\cL(K\cap[-3\ell,3\ell]^2)\geq\tfrac1{16}\log\ell$ still holds with probability at least~$\varepsilon/2$ under~$\nu_{A,\Gamma}$, provided that~$A\subset K\cup[-h,h]^2$ for some~$h$ much smaller than~$\ell$. Its proof adds to~$A$ the box~$[-h,h]^2$ and a path from the box to a vertex~$w\in\partial[-2h,2h]^2$ next to~$\gamma$, and applies Proposition~\ref{fo:local-comparison} at~$w$. Lemma~\ref{ct:needle} chooses the path so that the box and the path have in mean at most~$C\log h$ edges to vertices joined to~$\Gamma$, which is small compared with~$\log\ell$.

\begin{lemma}\label{ct:needle}
There is a universal constant~$C<\infty$ such that, for every integer~$h\geq10$ and every~$w\in\partial[-2h,2h]^2$, there is a nearest-neighbor path in~$[-2h,2h]^2$ from~$w$ to~$[-h,h]^2$ whose vertex set~$\pi$ satisfies~$\mathcal C([-h,h]^2\cup\pi\leftrightarrow\Z^2\setminus[-2h,2h]^2)\leq C\log h$.
\end{lemma}

\begin{proof}
Consider a nearest-neighbor path that starts at~$w$ and, until it enters~$[-h,h]^2$, steps toward~$0$ in a coordinate of largest absolute value. Along this path~$|\cdot|_\infty$ never increases and decreases at least once in every two consecutive steps. Let~$\pi$ be the vertex set of this path.

We bound the conductance by applying~\eqref{pre:dirichlet} to the function~$f$ defined for~$x\in\Z^2$ by
\begin{equation*}
f(x)\coloneqq\min\Bigl\{1,\frac{3\max\{0,2h+1-|x|_\infty\}}{1+|x-w|_\infty}\Bigr\}\,.
\end{equation*}
Then~$f$ vanishes outside~$[-2h,2h]^2$ and equals one on~$[-h,h]^2$ and, by the previous paragraph, on~$\pi$. Since~$|w|_\infty=2h$, the ratio in the definition of~$f$ is at most three, so~$f$ changes by at most~$6/(k+1)$ along a lattice edge~$xy$ with~$\max\{|x-w|_\infty,|y-w|_\infty\}=k$. At most~$32k$ edges have a given~$k\geq1$, so, since~$h\geq10$, the energy of~$f$ is at most~$\sum_{k=1}^{4h+1}1152k/(k+1)^2\leq2304\log h$.
\end{proof}

\begin{lemma}\label{ct:revealed-roots}
There is a universal constant~$\kappa\in\mathbb N$ such that the following holds, with~$\varepsilon$ as in Proposition~\ref{fo:local-comparison}. Let~$K$ and~$\gamma$ be as in Proposition~\ref{ct:reciprocal}, with~$0\in K$ and with a neighbor of~$0$ in~$\gamma$, and let~$h\geq10$ and~$\ell\geq h^\kappa$ be integers with~$\diam_\infty K\geq16\ell$. Let~$A$ and~$\Gamma$ be disjoint sets of vertices with~$K\subset A\subset K\cup [-h,h]^2$ and~$\gamma\subset\Gamma$. Then~$\nu_{A,\Gamma}\bigl(\cL(K\cap [-3\ell,3\ell]^2)\geq\tfrac1{16}\log \ell\bigr)\geq\varepsilon/2$.
\end{lemma}

\begin{proof}
Let~$C$ be the constant of Lemma~\ref{ct:needle} and~$\ell_*$ that of Proposition~\ref{fo:local-comparison}, and let~$\kappa$ be an integer with~$\kappa\geq\max\{\log_{10}\ell_*,32C/\varepsilon\}$. Since~$\gamma$ is connected with finite complement, there is a simple path in~$\gamma$ from a neighbor of~$0$ to the infinite component~$\Gamma'$ of~$\gamma\setminus [-2h,2h]^2$. Let~$w$ be the last vertex of this path in~$[-2h,2h]^2$, so that~$w\in\partial[-2h,2h]^2$ and~$\Gamma'$ contains a neighbor of~$w$. Let~$\pi$ be the vertex set of the path given by Lemma~\ref{ct:needle} for~$h$ and~$w$, and let~$A'\coloneqq K\cup[-h,h]^2\cup\pi$. Since~$A\subset A'$ and~$\Gamma'\subset\Gamma$, Lemma~\ref{fo:lattice-monotonicity} shows that~$\nu_{A',\Gamma'}$ is stochastically dominated by~$\nu_{A,\Gamma}$.

The sets~$A'$ and~$\Gamma'$ satisfy the assumptions of Proposition~\ref{fo:local-comparison} translated by~$w$, because~$\diam_\infty K\geq16\ell$ and~$\Gamma'$ is unbounded. Since~$\ell\geq h^\kappa\geq\ell_*$, this proposition gives
\begin{equation*}
\nu_{A',\Gamma'}\bigl(\cL(A'\cap (w+[-2\ell,2\ell]^2))\geq\tfrac18\log \ell\bigr)\geq\varepsilon\,.
\end{equation*}
Since~$A'\setminus K\subset[-h,h]^2\cup\pi\subset A'$ and~$\Gamma'$ lies outside~$[-2h,2h]^2$, the displays~\eqref{fo:color-marginal} and~\eqref{pre:count-conductance} bound the mean of~$\cL(A'\setminus K)$ under~$\nu_{A',\Gamma'}$ by~$\mathcal C([-h,h]^2\cup\pi\leftrightarrow\Z^2\setminus[-2h,2h]^2)$, which is at most~$C\log h$ by Lemma~\ref{ct:needle}. Markov's inequality and~$\ell\geq h^\kappa$ give
\begin{equation*}
\nu_{A',\Gamma'}\bigl(\cL(A'\setminus K)>\tfrac1{16}\log \ell\bigr)\leq\frac{16C\log h}{\log \ell}\leq\frac{16C}\kappa\leq\frac\varepsilon2\,.
\end{equation*}
Since~$|w|_\infty=2h\leq \ell$, the box~$w+[-2\ell,2\ell]^2$ lies in~$[-3\ell,3\ell]^2$, so~$\cL(A'\cap (w+[-2\ell,2\ell]^2))\leq\cL(K\cap [-3\ell,3\ell]^2)+\cL(A'\setminus K)$. The last two displays therefore show that~$\cL(K\cap [-3\ell,3\ell]^2)\geq\tfrac1{16}\log \ell$ with~$\nu_{A',\Gamma'}$-probability at least~$\varepsilon/2$. Since~$\nu_{A',\Gamma'}$ is stochastically dominated by~$\nu_{A,\Gamma}$, the same holds under~$\nu_{A,\Gamma}$, which proves the lemma.
\end{proof}

\subsection{Exploring the forest scale by scale}\label{ssec:scale-by-scale}

\begin{proof}[Proof of Proposition~\ref{ct:reciprocal}]
Let~$K$ and~$\gamma$ be as in Proposition~\ref{ct:reciprocal}, and translate so that~$0\in K$ and~$\gamma$ contains a neighbor of~$0$. Let
\begin{equation*}
\ell_i\coloneqq2^{(6\kappa)^i}\qquad\text{for }i\geq1\,,
\end{equation*}
with~$\kappa$ as in Lemma~\ref{ct:revealed-roots}. Assume first that~$\diam_\infty K\geq16\ell_2$, let~$n\geq2$ be the largest integer with~$16\ell_n\leq\diam_\infty K$, and let~$K_i\coloneqq K\cap[-3\ell_i,3\ell_i]^2$ for~$1\leq i\leq n$.

Run Wilson's algorithm of Section~\ref{ssec:prelim-wilson} with roots~$K\cup\gamma$ and an enumeration of the vertices that lists the vertices of~$[-3\ell_1-1,3\ell_1+1]^2$ first, then, for~$i=2,\ldots,n$ in turn, the remaining vertices of~$[-3\ell_i-1,3\ell_i+1]^2$, and then the remaining vertices. The set of vertices joined to~$\gamma$ then has law~$\nu_{K,\gamma}$. We call the walks started in~$[-3\ell_i-1,3\ell_i+1]^2$ the \textbf{walks up to scale}~$\ell_i$. These walks determine~$\cL(K_i)$. Given these walks, let~$A$ be the union of~$K$ and the vertices of the branches already built that are joined to~$K$, and let~$\Gamma$ be the union of~$\gamma$ and the vertices of the other branches already built. The remaining walks carry out Wilson's algorithm with roots~$A\cup\Gamma$, so, given the walks up to scale~$\ell_i$, the set of vertices joined to~$\gamma$, which is the set of vertices joined to~$\Gamma$, has conditional law~$\nu_{A,\Gamma}$.

\emph{Step 1.} We bound the probability that some walk up to scale~$\ell_i$ leaves~$[-\ell_i^6,\ell_i^6]^2$: for~$1\leq i<n$,
\begin{equation}\label{ct:beurling}
\PP\bigl(\text{some walk up to scale~$\ell_i$ leaves~$[-\ell_i^6,\ell_i^6]^2$}\bigr)\leq C\ell_i^{-\nf12}\,.
\end{equation}
Since~$0\in K$ and~$\diam_\infty K\geq16\ell_n\geq2\ell_i^6$, the connected set~$K$ contains a simple path from~$0$ to~$\partial[-\ell_i^6,\ell_i^6]^2$. A walk up to scale~$\ell_i$ starts within Euclidean distance~$6\ell_i+4$ of~$0$ and stops by its first visit to~$K$. So, given the earlier walks, the discrete Beurling estimate of~\citet[display~(1.6)]{Kesten1987} in the form of~\citet[Lemma~2.1, first bound]{Lawler2020}, applied to this path, bounds the probability that the walk leaves~$[-\ell_i^6,\ell_i^6]^2$ by~$C\ell_i^{-\nf52}$. Summing over the at most~$(6\ell_i+3)^2$ walks up to scale~$\ell_i$ gives~\eqref{ct:beurling}.

\emph{Step 2.} We bound the mean of~$\cL(K_i)^{-1}$ in terms of the mean of~$\cL(K_{i-1})^{-1}$: for~$2\leq i\leq n$,
\begin{equation}\label{ct:recursion}
\EE\bigl[\cL(K_i)^{-1}\bigr]\leq(1-\varepsilon/2)\,\EE\bigl[\cL(K_{i-1})^{-1}\bigr]+\frac{16}{\log\ell_i}+C\ell_{i-1}^{-\nf12}\,.
\end{equation}
Since~$K_{i-1}\subset K_i$ and~$0\in K_{i-1}$ has a neighbor in~$\gamma$, we have~$\cL(K_i)^{-1}\leq\cL(K_{i-1})^{-1}\leq1$. Suppose that the walks up to scale~$\ell_{i-1}$ stay in~$[-\ell_{i-1}^6,\ell_{i-1}^6]^2$. Then the root sets after these walks satisfy~$A\subset K\cup [-\ell_{i-1}^6,\ell_{i-1}^6]^2$, so Lemma~\ref{ct:revealed-roots} applies to the conditional law~$\nu_{A,\Gamma}$ given these walks, with~$h=\ell_{i-1}^6$ and~$\ell=\ell_i=h^\kappa$. So, given these walks, $\cL(K_i)^{-1}\leq16/\log\ell_i$ with probability at least~$\varepsilon/2$, and the conditional mean of~$\cL(K_i)^{-1}$ is at most~$(1-\varepsilon/2)\cL(K_{i-1})^{-1}+16/\log\ell_i$. By~\eqref{ct:beurling}, some walk up to scale~$\ell_{i-1}$ leaves~$[-\ell_{i-1}^6,\ell_{i-1}^6]^2$ with probability at most~$C\ell_{i-1}^{-\nf12}$, and taking expectations gives~\eqref{ct:recursion}.

\emph{Step 3.} We iterate~\eqref{ct:recursion} and deduce the proposition: with~$a\coloneqq-\log(1-\varepsilon/2)/\log(6\kappa)>0$, we show that
\begin{equation*}
\EE_{\nu_{K,\gamma}}\bigl[\cL(K)^{-1}\bigr]\leq C\bigl(\log(\diam_\infty K+2)\bigr)^{-a}\,.
\end{equation*}
Since~$\log\ell_i=(6\kappa)^i\log2$, the last two terms in~\eqref{ct:recursion} are at most~$C(6\kappa)^{-i}$, and since~$(6\kappa)^{-1}<1-\varepsilon/2$, iterating~\eqref{ct:recursion} from~$\cL(K_1)^{-1}\leq1$ gives~$\EE_{\nu_{K,\gamma}}[\cL(K)^{-1}]\leq\EE[\cL(K_n)^{-1}]\leq C(1-\varepsilon/2)^n=C(6\kappa)^{-an}$. By the choice of~$n$, we have~$\log(\diam_\infty K+2)\leq C(6\kappa)^n$, which proves the bound. Since~$\cL(K)\geq1$, increasing~$C$ extends the bound to~$\diam_\infty K<16\ell_2$.
\end{proof}

\section{A second moment for boundary edges}\label{sec:electrical}

In this section we prove Proposition~\ref{el:reference-moment}. We first recall its setting, drawn in Figure~\ref{fig:network}. The set~$D\subset\Z^2$ is finite and connected with connected complement, the set~$P\subset D$ is nonempty and connected, and~$e$ is a lattice edge from~$D^c$ to~$P$. The planar multigraph~$H$ is obtained by identifying~$D^c$ to a vertex~$\alpha$ and~$P$ to a vertex~$\beta$ and deleting the resulting loops, and~$H_0$ is~$H$ with~$e$ deleted. The count~$\cL_{H_0}$ is the number of edges of~$H_0$ from~$\alpha$ to vertices joined to~$\beta$ in the uniform rooted spanning forest of~$H_0$ with roots~$\alpha$ and~$\beta$. The function~$\phi$ is the voltage in~$H_0$ that equals one at~$\beta$ and zero at~$\alpha$, and~$m$ is the number of edges of~$H_0$ at~$\alpha$. We must show that~$\EE[\cL_{H_0}]$ is the energy of~$\phi$ and that
\begin{equation*}
\operatorname{Var}(\cL_{H_0})\leq C\,\EE[\cL_{H_0}]\Bigl(1+\log\frac m{\EE[\cL_{H_0}]}\Bigr)
\end{equation*}
for a universal constant~$C<\infty$.

By planar duality, $\cL_{H_0}$ is the number of edges that a path in a uniform spanning tree of a dual graph shares with a fixed path on the boundary of this graph, so~$\cL_{H_0}$ is a sum of indicators, one for each edge of~$H_0$ at~$\alpha$. In Section~\ref{ssec:boundary-path} we express the means of these indicators, and with the identities of~\citet*{KenyonWilson2009} the means of their pairwise products, through the Green function of the dual graph. This reduces the variance bound of Proposition~\ref{el:reference-moment} to a bound on differences of values of this Green function at the vertices of the boundary path, which we prove in Section~\ref{ssec:electrical-coordinates} using the square tiling of~\citet*{BrooksSmithStoneTutte1940} in the form of~\citet[Theorem~6.1]{Lovasz2019}.

\subsection{Edges on a boundary path}\label{ssec:boundary-path}

Going around~$\alpha$, starting after~$e$, we meet faces~$v_0,\ldots,v_m$ of~$H$ and the edges~$e_0,\ldots,e_{m-1}$ of~$H_0$ at~$\alpha$ in the order
\begin{equation*}
v_0,\ e_0,\ v_1,\ e_1,\ \ldots,\ e_{m-1},\ v_m\,,
\end{equation*}
and then~$e$, as in Figure~\ref{fig:dual-path}. Since~$D$ is connected, so is~$H\setminus\{\alpha\}$, and hence the faces~$v_0,\ldots,v_m$ are distinct. So the duals of~$e_0,\ldots,e_{m-1}$ form a path~$v_0,\ldots,v_m$ in the plane dual of~$H$, and together with the dual of~$e$, which joins~$v_m$ to~$v_0$, they form a cycle around~$\alpha$. Let~$H^\dagger$ be the plane dual of~$H$ with the dual of~$e$ deleted. Its vertices are the faces of~$H$. Then~$H^\dagger$ lies in the closed disk that this cycle bounds on the side away from~$\alpha$, and the path~$v_0,\ldots,v_m$ lies on the boundary of this disk.

\begin{figure}[htbp]
\centering
\includegraphics{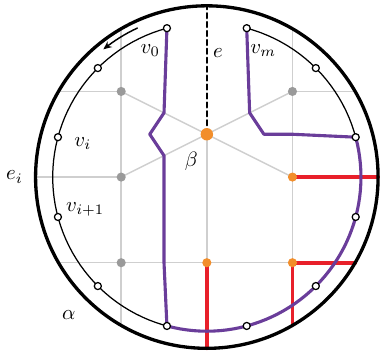}
\caption{The network~$H_0$ of Figure~\ref{fig:network}, with~$\alpha$ drawn as the circle. Going around~$\alpha$ from~$e$, the faces~$v_0,\ldots,v_m$, drawn as dots, alternate with the edges~$e_0,\ldots,e_{m-1}$. For the forest of Figure~\ref{fig:network}, the purple path from~$v_0$ to~$v_m$ in the corresponding spanning tree of~$H^\dagger$ uses the edge from~$v_i$ to~$v_{i+1}$ exactly when~$e_i$ is red.}
\label{fig:dual-path}
\end{figure}

Adding~$e$ to the uniform rooted spanning forest of~$H_0$ gives a tree uniformly distributed over the spanning trees of~$H$ that contain~$e$, and by~\citet[Proposition~12.1]{BLPS2001},
\begin{equation*}
\text{the duals of the edges absent from this tree form a uniform spanning tree of~$H^\dagger$}\,.
\end{equation*}
By planar duality, the path from~$v_0$ to~$v_m$ in this uniform spanning tree of~$H^\dagger$ consists of the duals of the edges of~$H_0$ between the two components of the forest. For~$0\leq i<m$, let~$I_i$ be the indicator that this path uses the dual of~$e_i$, that is, that the endpoint of~$e_i$ other than~$\alpha$ is joined to~$\beta$. Then
\begin{equation*}
\cL_{H_0}=\sum_{i=0}^{m-1}I_i\,.
\end{equation*}

Let~$\mathcal G$ be the Green function of~$H^\dagger$ grounded at~$v_m$, so that~$\mathcal G(x,y)$ is the voltage at~$y$ when unit current flows from~$x$ to~$v_m$ and the voltage at~$v_m$ is zero. In particular, $\mathcal G(x,v_m)=\mathcal G(v_m,x)=0$ for every vertex~$x$ of~$H^\dagger$, and~$\mathcal G(x,x)$ is the effective resistance between~$x$ and~$v_m$.

Lemma~\ref{el:means} below computes the means of the~$I_i$ and shows that~$\EE[\cL_{H_0}]$ is the energy of~$\phi$. Lemma~\ref{el:pair-contact} expresses the means of the products~$I_iI_j$ through~$\mathcal G$, and Lemma~\ref{el:arc-moment} then bounds~$\operatorname{Var}(\cL_{H_0})$ in terms of the differences~$\mathcal G(v_i,v_i)-\mathcal G(v_0,v_i)$. Lemma~\ref{el:split-estimate} in Section~\ref{ssec:electrical-coordinates} bounds each of these differences with~$0<i<m$ by a constant times~$1+\log(2/\EE[I_i])$ and by a constant times~$1+\log(2/\EE[I_{i-1}])$.

\begin{lemma}\label{el:means}
Every edge~$xy$ of~$H_0$, with faces~$v$ and~$v'$ of~$H$ on either side, satisfies
\begin{equation}\label{el:conjugate}
|\mathcal G(v_0,v)-\mathcal G(v_0,v')|=|\phi(x)-\phi(y)|\,.
\end{equation}
For~$0\leq i<m$, the endpoint~$x$ of~$e_i$ other than~$\alpha$ satisfies
\begin{equation}\label{el:mean-edge}
\mathcal G(v_0,v_i)-\mathcal G(v_0,v_{i+1})=\phi(x)=\EE[I_i]\,.
\end{equation}
Moreover,
\begin{equation}\label{el:mean-total}
\EE[\cL_{H_0}]=\mathcal G(v_0,v_0)\,,
\end{equation}
and this is the energy of~$\phi$.
\end{lemma}

\begin{proof}
By Kirchhoff's node law, $\phi$ has a harmonic conjugate on the faces of~$H$, which increases by the current through each edge of~$H_0$ from the face on the right of the edge to the face on the left, when the edge is oriented along the current. By Kirchhoff's cycle law around each face of~$H$, the harmonic conjugate is harmonic on~$H^\dagger$ except at~$v_0$ and~$v_m$, where unit current enters and leaves, because~$e$ is not an edge of~$H_0$ and~$\phi$ changes by one along~$e$. So, up to its sign and an additive constant, the harmonic conjugate is~$\mathcal G(v_0,\cdot)$, which proves~\eqref{el:conjugate}.

Applied to~$e_i$, whose faces are~$v_i$ and~$v_{i+1}$, the identity~\eqref{el:conjugate} gives~$|\mathcal G(v_0,v_i)-\mathcal G(v_0,v_{i+1})|=\phi(x)$, where~$x$ is the endpoint of~$e_i$ other than~$\alpha$. These differences have a common sign, since the currents through~$e_0,\ldots,e_{m-1}$ all flow into~$\alpha$, and their sum~$\mathcal G(v_0,v_0)$ is positive, so~$\mathcal G(v_0,v_i)-\mathcal G(v_0,v_{i+1})=\phi(x)$. By~\eqref{pre:wilson} and~\eqref{pre:voltage}, $\phi(x)$ is the probability that~$x$ is joined to~$\beta$, which proves~\eqref{el:mean-edge}.

Summing~\eqref{el:mean-edge} over~$i$ gives~\eqref{el:mean-total}. By~\eqref{pre:voltage} and~\eqref{pre:count-conductance}, the sum of~$\phi(x)$ over the endpoints~$x$ of~$e_0,\ldots,e_{m-1}$ other than~$\alpha$ is the energy of~$\phi$.
\end{proof}

To bound the variance of~$\cL_{H_0}$, we next express the means of the products~$I_iI_j$ through~$\mathcal G$.

\begin{lemma}\label{el:pair-contact}
For every~$0\leq i<j<m$ we have
\begin{align}\label{el:pair-expanded}
 \EE[I_iI_j]
 &=\mathcal G(v_0,v_i)\bigl(\mathcal G(v_{i+1},v_j)-\mathcal G(v_{i+1},v_{j+1})\bigr)
   -\mathcal G(v_0,v_{i+1})\bigl(\mathcal G(v_i,v_j)-\mathcal G(v_i,v_{j+1})\bigr)\notag\\
 &\quad-\mathcal G(v_0,v_j)\bigl(\mathcal G(v_i,v_{j+1})-\mathcal G(v_{i+1},v_{j+1})\bigr)
   +\mathcal G(v_0,v_{j+1})\bigl(\mathcal G(v_i,v_j)-\mathcal G(v_{i+1},v_j)\bigr)\,.
\end{align}
\end{lemma}

\begin{proof}
The product~$I_iI_j$ is the indicator that the path from~$v_0$ to~$v_m$ in the uniform spanning tree of~$H^\dagger$ uses the duals of both~$e_i$ and~$e_j$. We write its mean as a ratio of numbers of spanning forests of~$H^\dagger$ and compute this ratio with the identities of~\citet*{KenyonWilson2009}.

Assume first that~$v_m,v_0,v_i,v_{i+1},v_j,v_{j+1}$ are distinct, and call them the \textbf{nodes}~1 to~6, in this order, which is their order along the cycle formed by the path~$v_0,\ldots,v_m$ and the dual of~$e$. Since~$H^\dagger$ lies in the disk bounded by this cycle, two disjoint paths in~$H^\dagger$ cannot connect two pairs of vertices that alternate along the cycle~\citep[Section~1.1]{KenyonWilson2011}. Hence a simple path from~$v_0$ to~$v_m$ in~$H^\dagger$ traverses the dual of each~$e_k$ only from~$v_k$ to~$v_{k+1}$ and traverses these duals in increasing order of~$k$.

A \textbf{grove} is a spanning forest of~$H^\dagger$ in which every component contains a node~\citep[Section~1.1]{KenyonWilson2011}. For a partition~$\sigma$ of the nodes into pairs, let~$Z(\sigma)$ be the number of groves with partition~$\sigma$, that is, whose components contain the pairs of~$\sigma$, one pair each. By the order of traversal above, deleting the duals of~$e_i$ and~$e_j$ from a spanning tree of~$H^\dagger$ whose path from~$v_0$ to~$v_m$ uses both of them leaves a grove with partition~$16|23|45$, and adding the two duals to such a grove gives such a tree. Hence
\begin{equation*}
\EE[I_iI_j]=\frac{Z(16|23|45)}{\lvert\operatorname{SpanT}(H^\dagger)\rvert}\,.
\end{equation*}

To compute~$Z(16|23|45)$, we use an identity of~\citet*[Theorems~8.1 and~B.3]{KenyonWilson2009}. List two disjoint triples of nodes in increasing order, the first containing node~1 and the second containing node~6. Let~$\sigma$ run over the partitions of the nodes into pairs that each consist of a node of the first triple and a node of the second, and let~$\operatorname{sgn}(\sigma)$ be the sign of the permutation that~$\sigma$ induces between the two lists. Then
\begin{equation*}
\sum_\sigma\operatorname{sgn}(\sigma)\frac{Z(\sigma)}{\lvert\operatorname{SpanT}(H^\dagger)\rvert}
\end{equation*}
is the coefficient of~$\lambda$ in the polynomial
\begin{equation*}
\det\Bigl(\lambda-\tfrac12\bigl(\mathcal G(x,x)+\mathcal G(y,y)-2\mathcal G(x,y)\bigr)\Bigr)_{x,y}\,,
\end{equation*}
where~$x$ runs over the first triple and~$y$ over the second. Since node~1 is~$v_m$ and~$\mathcal G(v_m,\cdot)=0$, subtracting the row of node~1 from the other rows and the column of node~6 from the other columns shows that this coefficient is the two-by-two determinant with entries~$\mathcal G(x,y)-\mathcal G(x,v_{j+1})$, where~$x$ runs over the first triple without node~1 and~$y$ over the second triple without node~6.

By planarity, $Z(\sigma)=0$ when two pairs of~$\sigma$ alternate along the cycle. So for the triples~$124$ and~$356$ the sum is~$\bigl(Z(16|23|45)-Z(16|25|34)\bigr)/\lvert\operatorname{SpanT}(H^\dagger)\rvert$, and for the triples~$123$ and~$456$ it is~$-Z(16|25|34)/\lvert\operatorname{SpanT}(H^\dagger)\rvert$. Subtracting the second from the first and using that node~2 is~$v_0$ gives
\begin{align*}
\EE[I_iI_j]
&=\det\begin{pmatrix}
\mathcal G(v_0,v_i)-\mathcal G(v_0,v_{j+1})&\mathcal G(v_0,v_j)-\mathcal G(v_0,v_{j+1})\\
\mathcal G(v_{i+1},v_i)-\mathcal G(v_{i+1},v_{j+1})&\mathcal G(v_{i+1},v_j)-\mathcal G(v_{i+1},v_{j+1})
\end{pmatrix}\\
&\quad-\det\begin{pmatrix}
\mathcal G(v_0,v_{i+1})-\mathcal G(v_0,v_{j+1})&\mathcal G(v_0,v_j)-\mathcal G(v_0,v_{j+1})\\
\mathcal G(v_i,v_{i+1})-\mathcal G(v_i,v_{j+1})&\mathcal G(v_i,v_j)-\mathcal G(v_i,v_{j+1})
\end{pmatrix}\,,
\end{align*}
and expanding the determinants gives~\eqref{el:pair-expanded}.

When~$i=0$ or~$j=i+1$ or~$j=m-1$, two of the six vertices coincide. For each such coincidence, we attach to the shared vertex a new vertex of degree one outside the disk and use it as node~3 or~5 or~6, respectively. Off the diagonal, the Green function at the new vertex equals the Green function at the shared vertex, so~\eqref{el:pair-expanded} still holds.
\end{proof}

Summing~\eqref{el:pair-expanded} over the pairs~$i<j$, we now bound the variance of~$\cL_{H_0}$ in terms of the differences~$\mathcal G(v_k,v_k)-\mathcal G(v_0,v_k)$.

\begin{lemma}\label{el:arc-moment}
We have
\begin{equation*}
 \operatorname{Var}(\cL_{H_0})\leq\EE[\cL_{H_0}]+3\sum_{i=0}^{m-1}\EE[I_i]\bigl(\mathcal G(v_i,v_i)-\mathcal G(v_0,v_i)+\mathcal G(v_{i+1},v_{i+1})-\mathcal G(v_0,v_{i+1})\bigr)\,.
\end{equation*}
\end{lemma}

\begin{proof}
We sum~\eqref{el:pair-expanded} over the pairs~$i<j$. When we sum the first two terms over~$j$ and the last two over~$i$, the Green-function differences telescope because~$\mathcal G(\cdot,v_m)=0$, and with~\eqref{el:mean-edge} this gives the first equality in
\begin{align*}
 \sum_{i<j}\EE[I_iI_j]
 &=\sum_{i=0}^{m-1}\Bigl(\mathcal G(v_0,v_i)\,\mathcal G(v_{i+1},v_{i+1})-\mathcal G(v_0,v_{i+1})\,\mathcal G(v_i,v_i)+\EE[I_i]\,\mathcal G(v_i,v_{i+1})\Bigr)\\
 &=\sum_{i=0}^{m-1}\EE[I_i]\Bigl(\mathcal G(v_i,v_i)-\mathcal G(v_0,v_i)+\mathcal G(v_{i+1},v_{i+1})-\mathcal G(v_0,v_{i+1})+\mathcal G(v_i,v_{i+1})\Bigr)\,.
\end{align*}
The second equality is a summation by parts, since~$\mathcal G(v_k,v_k)-\mathcal G(v_0,v_k)$ vanishes for~$k=0$ and for~$k=m$.

Since~$\mathcal G$ is positive semidefinite, we have~$2\mathcal G(v_i,v_{i+1})\leq\mathcal G(v_i,v_i)+\mathcal G(v_{i+1},v_{i+1})$, and by~\eqref{el:mean-edge} we have~$\sum_i\EE[I_i]\bigl(\mathcal G(v_0,v_i)+\mathcal G(v_0,v_{i+1})\bigr)=\mathcal G(v_0,v_0)^2$. Hence
\begin{equation*}
 \sum_{i<j}\EE[I_iI_j]\leq\tfrac12\mathcal G(v_0,v_0)^2+\tfrac32\sum_{i=0}^{m-1}\EE[I_i]\bigl(\mathcal G(v_i,v_i)-\mathcal G(v_0,v_i)+\mathcal G(v_{i+1},v_{i+1})-\mathcal G(v_0,v_{i+1})\bigr)\,.
\end{equation*}
Since~$I_i^2=I_i$ and~$\EE[\cL_{H_0}]=\mathcal G(v_0,v_0)$ by~\eqref{el:mean-total}, we have
\begin{equation*}
 \operatorname{Var}(\cL_{H_0})=\mathcal G(v_0,v_0)+2\sum_{i<j}\EE[I_iI_j]-\mathcal G(v_0,v_0)^2\,,
\end{equation*}
and the lemma follows.
\end{proof}

\subsection{The square tiling}\label{ssec:electrical-coordinates}

In this section we bound the differences~$\mathcal G(v_i,v_i)-\mathcal G(v_0,v_i)$ and prove Proposition~\ref{el:reference-moment}.

We give each vertex~$x$ of~$H_0$ the horizontal coordinate~$\phi(x)$ and each face~$v$ of~$H$ the vertical coordinate~$-\mathcal G(v_0,v)$. For an edge~$xy$ of~$H_0$, the \textbf{square} of~$xy$ is the product of the interval between~$\phi(x)$ and~$\phi(y)$ and the interval between the vertical coordinates of the two faces of~$H$ on either side of~$xy$. By~\eqref{el:conjugate}, both intervals have length~$|\phi(x)-\phi(y)|$.

Deleting~$e$ from~$H$ merges the faces~$v_0$ and~$v_m$ into one face of~$H_0$, whose boundary contains~$\alpha$ and~$\beta$, and we take this face as the outer face of~$H_0$. Since~$\alpha$ and~$\beta$ then lie on the outer face of~$H_0$, the squares of the edges of positive current form the square tiling of~\citet*{BrooksSmithStoneTutte1940} that~\citet[Theorem~6.1 and its proof]{Lovasz2019} builds from~$H_0$. Figure~\ref{fig:tiling} shows this tiling.

\begin{figure}[htbp]
\centering
\includegraphics{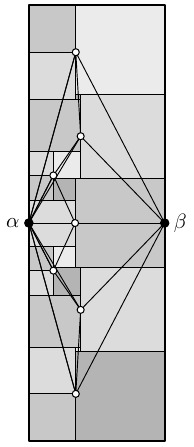}
\caption{The square tiling of the network~$H_0$ of Figure~\ref{fig:network}. Each vertex~$x$ lies on the vertical segment at horizontal coordinate~$\phi(x)$ along which the squares of its edges meet, with~$\alpha$ on the left and~$\beta$ on the right.}
\label{fig:tiling}
\end{figure}

Since~$\phi$ vanishes at~$\alpha$ and the faces on either side of~$e_k$ are~$v_k$ and~$v_{k+1}$, the identity~\eqref{el:mean-edge} shows that the square of~$e_k$ is
\begin{equation*}
[0,\EE[I_k]]\times[-\mathcal G(v_0,v_k),-\mathcal G(v_0,v_{k+1})]\qquad\text{for }0\leq k<m\,.
\end{equation*}
So the squares of the edges at~$\alpha$ are stacked along the vertical axis, with the square of~$e_0$ at the bottom and the square of~$e_{m-1}$ at the top.

For~$0<i<m$, let~$H_0'$ be the network obtained from~$H_0$ by splitting~$\alpha$ into a vertex~$\alpha_0$ incident to~$e_0,\ldots,e_{i-1}$ and a vertex~$\alpha_1$ incident to~$e_i,\ldots,e_{m-1}$, and extend~$\phi$ to~$H_0'$ by zero at~$\alpha_0$ and at~$\alpha_1$. Splitting~$\alpha$ splits the stack of squares at the height~$-\mathcal G(v_0,v_i)$, where the squares of~$e_{i-1}$ and~$e_i$ meet: the squares of the edges at~$\alpha_0$ lie below this height, and those at~$\alpha_1$ lie above it.

\begin{lemma}\label{el:split-estimate}
There is a universal constant~$C<\infty$ such that every~$0<i<m$ and~$j\in\{i-1,i\}$ satisfy
\begin{equation*}
 \mathcal G(v_i,v_i)-\mathcal G(v_0,v_i)\leq C\Bigl(1+\log\frac2{\EE[I_j]}\Bigr)\,.
\end{equation*}
\end{lemma}

\begin{proof}
The proof has four steps. Step~1 uses planar duality to bound~$\mathcal G(v_i,v_i)-\mathcal G(v_0,v_i)$ by the energy, with weight~$(1-\phi(x))(1-\phi(y))$ on each edge~$xy$, of any function on the vertices of~$H_0'$ that is zero at~$\alpha_0$ and one at~$\alpha_1$. Step~2 chooses a function~$F$ on the plane that is zero on the vertical axis below the square of~$e_j$ and one on the vertical axis above it, and whose oscillation over each square is at most a constant times the side length of the square divided by the larger of~$\EE[I_j]$ and the distance from the square to the square of~$e_j$. Step~3 reads off from~$F$ a function on the vertices of~$H_0'$ whose weighted energy is at most one plus a constant times the sum of the squared oscillations of~$F$ over the squares. Step~4 bounds this sum by a constant times~$1+\log(2/\EE[I_j])$, using that the squares have disjoint interiors and lie in a strip of width one. For a function~$F$ on the plane and a compact set~$Q$, let~$\operatorname{osc}_QF\coloneqq\max_QF-\min_QF$.

\emph{Step 1.} We bound~$\mathcal G(v_i,v_i)-\mathcal G(v_0,v_i)$ by a sum over the edges of~$H_0'$: every function~$f$ on the vertices of~$H_0'$ with~$f(\alpha_0)=0$ and~$f(\alpha_1)=1$ satisfies
\begin{equation}\label{el:weighted}
 \mathcal G(v_i,v_i)-\mathcal G(v_0,v_i)\leq\sum_{xy}(1-\phi(x))(1-\phi(y))(f(x)-f(y))^2\,,
\end{equation}
where the sum runs over the edges of~$H_0'$.

By planar duality, the effective resistance~$\mathcal G(v_i,v_i)$ between~$v_i$ and~$v_m$ in~$H^\dagger$ is an effective conductance in~$H_0'$. Indeed, adding to~$H^\dagger$ an edge from~$v_i$ to~$v_m$ that bounds a face with the path~$v_i,\ldots,v_m$ gives a graph whose plane dual is~$H_0'$ with~$\alpha_0$ and~$\beta$ identified and with an edge from~$\alpha_1$ to the resulting vertex, so~\citet[Proposition~9.4]{LyonsPeres2016} gives
\begin{equation*}
\mathcal G(v_i,v_i)=\mathcal C(\alpha_1\leftrightarrow\{\alpha_0,\beta\})\quad\text{in~$H_0'$}\,.
\end{equation*}

By~\eqref{pre:dirichlet}, this conductance is at most the energy of~$(1-\phi)f$, which is one at~$\alpha_1$ and zero at~$\alpha_0$ and~$\beta$. Every edge~$xy$ satisfies
\begin{multline*}
 \bigl((1-\phi(x))f(x)-(1-\phi(y))f(y)\bigr)^2=(1-\phi(x))(1-\phi(y))(f(x)-f(y))^2\\
 +(\phi(y)-\phi(x))\bigl((1-\phi(x))f(x)^2-(1-\phi(y))f(y)^2\bigr)\,.
\end{multline*}
Summed over the edges, the last terms give the sum over the vertices~$x$ of~$(1-\phi(x))f(x)^2$ times the net current into~$x$. Only~$\alpha_1$ contributes to this sum, because~$\phi$ is harmonic at the vertices of~$D\setminus P$ and~$(1-\phi)f$ vanishes at~$\alpha_0$ and~$\beta$. By~\eqref{el:mean-edge}, the net current into~$\alpha_1$ is~$\sum_{k=i}^{m-1}\EE[I_k]=\mathcal G(v_0,v_i)$. Hence the energy of~$(1-\phi)f$ is~$\mathcal G(v_0,v_i)$ plus the right side of~\eqref{el:weighted}, which proves~\eqref{el:weighted}.

\emph{Step 2.} We place the squares in a strip and choose a function~$F$ on the plane such that every axis-parallel square~$Q$ satisfies
\begin{equation}\label{el:max-angular-osc}
 \operatorname{osc}_QF\leq\frac{\diam_\infty Q}{\max\{1,\min_Q|\cdot|_\infty\}}\,.
\end{equation}
Shift the vertical coordinate and multiply both coordinates by~$2/\EE[I_j]$, so that the square of~$e_j$ becomes~$[0,2]\times[-1,1]$. Then the squares lie in the strip~$[0,2/\EE[I_j]]\times\mathbb R$, the squares of the edges at~$\alpha_0$ other than~$e_j$ lie below height~$-1$, and those of the edges at~$\alpha_1$ other than~$e_j$ lie above height~$1$. For a point~$z=(z_1,z_2)$ of the plane, let
\begin{equation*}
 F(z)\coloneqq\frac12+\frac{z_2}{2\max\{1,|z|_\infty\}}\,.
\end{equation*}
Then~$0\leq F\leq1$, the function~$F(0,\cdot)$ vanishes on~$(-\infty,-1]$ and equals one on~$[1,\infty)$, and~\eqref{el:max-angular-osc} holds because~$|z_2|\leq\max\{1,|z|_\infty\}$ and both~$z_2$ and~$\max\{1,|z|_\infty\}$ are~$1$-Lipschitz in~$z$ for~$|\cdot|_\infty$.

\emph{Step 3.} We choose~$f$ and compare it with~$F$, to show that
\begin{equation}\label{el:transfer}
 \sum_{xy}(1-\phi(x))(1-\phi(y))(f(x)-f(y))^2\leq1+21\sum_Q\bigl(\operatorname{osc}_QF\bigr)^2\,,
\end{equation}
where~$Q$ ranges over the squares of the edges. Let~$f(\alpha_0)\coloneqq0$ and~$f(\beta)\coloneqq0$, and let~$f(\alpha_1)\coloneqq1$. For~$x\in D\setminus P$, let~$f(x)$ be the value of~$F$ at the point with the horizontal coordinate of~$x$ and the smallest vertical coordinate of the faces at~$x$, as in~\citet[Lemma~5.2]{BenjaminiSchramm1996}.

Fix an edge~$xy$ of~$H_0'$ that is neither~$e_j$ nor incident to~$\beta$, and let~$Q$ be its square. For an endpoint~$x\in D\setminus P$ of~$xy$, since~$x$ has four edges, a face at~$x$ of smallest vertical coordinate either lies on one side of~$xy$ or shares another edge at~$x$ with a face on one side of~$xy$. In the first case, and for the endpoints~$\alpha_0$ and~$\alpha_1$, let~$Q_x$ be a corner of~$Q$, so that~$\operatorname{osc}_{Q_x}F=0$. In the second case, let~$Q_x$ be the square of this other edge. Define~$Q_y$ in the same way. Then
\begin{equation*}
 |f(x)-f(y)|\leq\operatorname{osc}_QF+\operatorname{osc}_{Q_x}F+\operatorname{osc}_{Q_y}F\,.
\end{equation*}
Indeed, by the choice of~$Q_x$ and the values of~$F$ on the vertical axis, $f(x)$ differs from the value of~$F$ at a corner of~$Q$ by at most~$\operatorname{osc}_{Q_x}F$, and similarly for~$y$. Squaring this bound and summing over the edges proves~\eqref{el:transfer}, since each square appears as~$Q$ or~$Q_x$ or~$Q_y$ for at most seven edges, the weights in~\eqref{el:transfer} are at most one and vanish on the edges at~$\beta$, and the term of~$e_j$ is at most one.

\emph{Step 4.} We sum the oscillations and deduce the lemma: we show that
\begin{equation}\label{el:osc-sum}
 \sum_Q\bigl(\operatorname{osc}_QF\bigr)^2\leq16\Bigl(1+\log\frac2{\EE[I_j]}\Bigr)\,.
\end{equation}
Let~$Q$ be a square. By~\eqref{el:max-angular-osc} and~$0\leq F\leq1$, every~$z\in Q$ satisfies
\begin{equation*}
 \operatorname{osc}_QF\leq\frac{2\diam_\infty Q}{\max\{1,|z|_\infty\}}\,,
\end{equation*}
and averaging the square of this bound over~$z\in Q$ gives
\begin{equation*}
 \bigl(\operatorname{osc}_QF\bigr)^2\leq4\int_Q\frac{dz}{\max\{1,|z|_\infty\}^2}\,.
\end{equation*}
The squares have disjoint interiors and lie in the strip~$[0,2/\EE[I_j]]\times\mathbb R$, and every~$z_1\geq0$ satisfies
\begin{equation*}
 \int_{-\infty}^\infty\frac{dz_2}{\max\{1,|(z_1,z_2)|_\infty\}^2}=\frac4{\max\{1,z_1\}}\,.
\end{equation*}
Since~$\EE[I_j]\leq1$, summing the bound on~$(\operatorname{osc}_QF)^2$ over the squares and integrating this identity over~$0\leq z_1\leq2/\EE[I_j]$ gives~\eqref{el:osc-sum}. Combining~\eqref{el:weighted}, \eqref{el:transfer}, and~\eqref{el:osc-sum} gives
\begin{equation*}
 \mathcal G(v_i,v_i)-\mathcal G(v_0,v_i)\leq1+21\cdot16\Bigl(1+\log\frac2{\EE[I_j]}\Bigr)\,,
\end{equation*}
which proves the lemma.
\end{proof}

\begin{proof}[Proof of Proposition~\ref{el:reference-moment}]
Lemma~\ref{el:means} shows that~$\EE[\cL_{H_0}]$ is the energy of~$\phi$. Let~$0\leq i<m$. Since~$\mathcal G(v_k,v_k)-\mathcal G(v_0,v_k)$ vanishes for~$k=0$ and for~$k=m$, Lemma~\ref{el:split-estimate}, applied to the pair~$(i,i)$ when~$i>0$ and to the pair~$(i+1,i)$ when~$i+1<m$, gives
\begin{equation*}
 \mathcal G(v_i,v_i)-\mathcal G(v_0,v_i)+\mathcal G(v_{i+1},v_{i+1})-\mathcal G(v_0,v_{i+1})\leq C\Bigl(1+\log\frac2{\EE[I_i]}\Bigr)\,.
\end{equation*}
Multiplying this bound by~$\EE[I_i]$, summing over~$i$, and applying Lemma~\ref{el:arc-moment} gives
\begin{equation*}
 \operatorname{Var}(\cL_{H_0})\leq\EE[\cL_{H_0}]+3C\sum_{i=0}^{m-1}\EE[I_i]\Bigl(1+\log\frac2{\EE[I_i]}\Bigr)\,.
\end{equation*}
Jensen's inequality for the concave logarithm with weights~$\EE[I_i]/\EE[\cL_{H_0}]$ gives
\begin{equation*}
 \sum_{i=0}^{m-1}\EE[I_i]\log\frac1{\EE[I_i]}\leq\EE[\cL_{H_0}]\log\frac m{\EE[\cL_{H_0}]}\,.
\end{equation*}
Since~$\EE[\cL_{H_0}]\leq m$, the last two displays prove~\eqref{el:reference-estimate}.
\end{proof}

\bibliographystyle{plainnat}
\bibliography{references}
\end{document}